\documentclass[a4paper]{article} 
\usepackage[left=2.3cm,right=2.3cm,top=2.9cm]{geometry}
\usepackage{authblk}
\usepackage{CJKutf8}
\usepackage{yfonts}
\usepackage{lipsum}
\usepackage[utf8]{inputenc}
\title{\makebox[1\textwidth][c]{Scaling limit of the step-reinforced and stochastic Lévy--Lorentz} model on weakly entangled integer lattice}
\author{Jiaming Chen\thanks{chen.jiaming@cims.nyu.edu}}
\affil{Courant Institute of Mathematical Sciences, New York University}
\date{\today}

\usepackage{graphicx}
\usepackage{amssymb}
\usepackage{amsmath}
\usepackage{amsthm}
\usepackage{tikz}
\usepackage{todonotes}
\usepackage{enumitem}
\usepackage{verbatim}

\usepackage{centernot}
\usepackage{mathtools}
\usepackage{ stmaryrd }

\usepackage{hyperref}
\hypersetup{allcolors=black,
pdftitle={Stochastic Levy-Lorentz gas},
pdfpagemode=FullScreen}
\numberwithin{equation}{section}
\usepackage{bbm}

\usepackage{braket}
\usepackage{amsmath,amsthm,amssymb}
\usepackage{physics}
\usepackage{mathtools}
\usepackage{bm}
\usepackage{esvect}
\usepackage[english]{babel}
\usepackage{extarrows} 
\usepackage{mathrsfs}
\usepackage{esint}

\usepackage{setspace}
\usepackage{titlesec}
\titleformat{\subsection}[runin]
  {\normalfont\large\bfseries}{\thesubsection}{1em}{}

\usepackage{amsmath}

\numberwithin{equation}{section}
\usepackage{bbm}

\usepackage{braket}
\usepackage{amsmath,amsthm,amssymb}
\usepackage{physics}
\usepackage{mathtools}
\usepackage{bm}
\usepackage{esvect}
\usepackage[english]{babel}

\newtheorem{theorem}{Theorem}[section]
\newtheorem{lemma}[theorem]{Lemma}

\theoremstyle{definition}

\theoremstyle{remark}
\newtheorem{remark}[theorem]{Remark}

\usepackage[english]{babel}

\newcommand\normx[1]{\lVert#1\rVert}
\newcommand\normy[1]{\big\lVert#1\big\rVert}

\makeatletter
\renewenvironment{proof}[1][\proofname]{%
  \par\pushQED{\qed}\normalfont%
  \topsep6\p@\@plus6\p@\relax
  \trivlist\item[\hskip\labelsep\bfseries#1\@addpunct{.}]%
  \ignorespaces
}{%
  \popQED\endtrivlist\@endpefalse
}
\makeatother
\begin{document}
\maketitle

\begin{abstract}
    This paper describes the stochastic Lévy--Lorentz gas driven by general long-range reference random walk on correlated and entangled random medium. Further consideration has been laid on the stochastic reinforcement of the underlying random walk, where it now possesses memory. Central limit theorems are obtained in both cases.
\end{abstract}


\tableofcontents

\section{Introduction}
    Shedding light onto a piece of glass, the reflections of the light particle perform Lévy flight \cite{Barthelemy/Bertolotti/Wiersma} on the micro-spheres of varying sizes in the glass. This interpretation of the motion of light surged from the designation of a crystal material \cite{Akhmerov/Beenakker/Groth} called Lévy glass, where the refracture of light is explained by the anomalous diffusion of random scattering field. Initiated by the study \cite{Barkai/Fleurov/Klafter}, such system is coined to the name \textit{Lévy--Lorentz gas} and is modeled via random walks on point processes \cite{Berger/Rosenthal}. A random array of points, called targets and with correlated distances, is given on $\mathbb{Z}$ with a particle traveling at normalized speed among the targets. Here the random walk takes value on the marked points instead of on $\mathbb{Z}$, and such motion is often called random walk in random scenery \cite{Hollander/Steif} as well. The inspiration from light in Lévy glass aside, the Lévy--Lorentz gas has no lack of motivation such as the transport of electrons in amorphous fluid \cite{Scher/Montroll} and the transmittent $N$-body Hamiltonian dynamics \cite{Antoni/Torcini}. Even so, it is still a fact that this topic is regrettably less developed than others in the realm of random walk. The Lévy--Lorentz gas, to the best of our knowledge, of which a rigorous mathematical treatment has been given only to its simplest form, where the targets keep i.i.d. distances and the random walk is Markovian. In this work, we extend the frontier of this study by allowing the targets to correlate with each other, see (\ref{eqn: kkkkkkk}), and the underlying random walk will be able to possess long-term memory which then forces it to become a reinforced random walk.\par
    The origin of the framework of the Lévy--Lorentz gas can be traced back to \cite{Barkai/Fleurov/Klafter} via anomalous diffusion, whereas the starting point of a rigorous mathematical investigation on the same topic is \cite{Bianchi/Cristadoro/Lenci/Ligabo}. Studies of such random walks in random scenery are usually divided into two categories: that of the quenched processes and that of the annealed processes. A random walk moving in the quenched random medium corresponds to the scenario where the disordered configuration changes {from the viewpoint of the particle}. And the motion of the quenched processes is Markovian but non-invariant w.r.t. shiftings. In theoretical physics, such quenched dynamics is used to describe aging \cite{Ben Arous/Cerny} of certain systems such as the Bouchaud model \cite{Bouchaud}. If we mod out the randomness in the disordered medium by integration, then we embrace the annealed processes which is invariant but non-Markovian, and often corresponds to the studies such as the fractional Fokker--Planck equation \cite{Metzler/Klafter}. In their work \cite{Bianchi/Cristadoro/Lenci/Ligabo}, the discrete-time Lévy--Lorentz gas, or random walk on point processes, normally rescales to a Gaussian in both the quenched and annealed distributions. The continuous-time Lévy--Lorentz gas was also introduced there by interpolating its traces between each two collision times, which guarantees the normalized speed of the continuous-time model. Following \cite{Magdziarz/Szczotka}, via a Skorokhod space approach, the discrete-time Lévy--Lorentz gas is shown to converge weakly to a Wiener process on the line. And in \cite{Zamparo}, the large fluctuation of the continuous-time Lévy--Lorentz gas as well as the resulting transport properties of the model were established. We should also mention that weak convergence and convergence in finite dimensional marginals have been discussed in \cite{Stivanello/Bet/Bianchi/Lenci/Magnanini} and also in the superdiffusive regime \cite{Bianchi/Lenci/Pene}. The interplay between a persistent random walk and the Lévy--Lorentz gas resides in \cite{Artuso/Cristadoro/Onofri/Radice}. There is also a concise summary \cite{Lenci} of recent results on Lévy--Lorentz gas, which is self-contained and written up in unified manner.\par
    The purpose of this work is to extend previous researches on the convergence in law to the scenario where the stochastic distances between each two scatterers in the Lévy--Lorentz gas are allowed to be dependent. Indeed, the i.i.d. assumption in previous works is too idealized and the mutual interaction on lattice prevails in nature: The quantum dissipation \cite{Yan/Pollet/Lou/Wang/Chen/Cai} of charges, the fermion localization \cite{Tusi/Fallani} and the boson correlation \cite{Yamashika/Kagamihara/Yoshii/Tsuchiya} in optical materials. Of sufficient generality and special interest is the random array provided by Gibbsian particle ensembles appeared in the context of spin-glass systems \cite{Rassoul-Agha} at high temperature, which satisfies the Dobrushin--Shlosman correlation condition \cite[p. 378]{Dobrushin/Shlosman}, see also \cite{Martinelli} for further reference and \cite{Comets/Zeitouni2} for related work. An important feature of such entanglement on the integer lattice $\mathbb{Z}$ is that the influence of the distances array in remote regions decays exponentially as the separation length grows. The approach to investigating such entangled lattice array usually appears when one sees the stochastic evolution of the Lévy--Lorentz gas from the viewpoint of the particle, as has been adopted in \cite{Bianchi/Cristadoro/Lenci/Ligabo}. Indeed, we shall reconstruct or give meaning to the convergence in distribution of the whole trajectories. Thus the present article is fully connected with the spirit of Feynman’s phrase: “There is pleasure in recognizing old things from a new viewpoint”, \cite{Feynman}. As a result of that recognition, we will further extend the frontier of knowledge by considering particles with long-time memory, which has not been discussed before in the literature of stochastic Lévy--Lorentz gas.\par\noindent
    \textbf{Acknowledgement.} I am grateful to my PhD Advisor Alejandro Ramírez for teaching me the crucial technique to resolve the correlated random medium. I am also very grateful to my PhD Advisor Prof. Dr. Pierre Tarrès for guiding me to the notion of reinforced random walk. And I acknowledge the anonymous Reviewer for evaluating our manuscript and for providing comments and suggestions which really helped to improve the paper substantially. I also need to thank the financial support from NYU-ECNU Institute of Mathematical Sciences.\par\noindent
    \textbf{Conflict of interest statement}. The author declares there is no conflict of interest.\par\noindent
    \textbf{Data availability statement}. There is no data used in this work.

\section{Step-reinforced and stochastic Lévy--Lorentz gas}
    The stochastic Lévy--Lorentz gas is described via an underlying random walk. In this section, we present the scaling limits of such Lévy--Lorentz gas model with the underlying random walk either Markovian or possessing long-term memory, i.e.~reinforcement. The lattice scatterers are aligned in disordered medium $\mathbb{Z}$ in the sense that they are allowed to be mutually interacting.

\subsection{Stochastic Lévy--Lorentz gas.} \label{sec: 2.1}

    The one-dimensional Lévy--Lorentz gas can be described as follows. Let $\omega=(\omega_r)_{r\in\mathbb{Z}}$ be an array of scatterers indexed by the integer lattice $\mathbb{Z}$ in ascending order. A particle with hardcore traverses on the line with either discrete-time or continuous-time motion successively collides with these scatterers. The law of the random scatterers $(\omega_r)_{r\in\mathbb{Z}}$ will be specified in the squeal.\par
    Following the preceding framework in \cite{Bianchi/Cristadoro/Lenci/Ligabo,Zamparo}, we assume that $V_1,V_2,\ldots$ are i.i.d. jumps on $\mathbb{Z}$ with distribution $(p_k)_{k\in\mathbb{Z}}$ satisfying $\sum_{k\in\mathbb{Z}}p_k=1$ and bounded in $L^1$, that is, $\sum_{k\in\mathbb{Z}}\abs{k}p_k<\infty$. The $n$th collision takes place at the scattering point labeled by $S_n\coloneqq\sum_{k=1}^n V_k$, for all $n\geq0$. And thus $(S_n)_{n\geq0}$ is a Markov random walk on $\mathbb{Z}$ that is allowed to be long-range, whose law is denoted as $P^S_0$ and $E^S_0\coloneqq E_{P^S_0}$.\par
    The motion of the particle traversing on the scatterers is described by a random walk $(X_n)_{n\geq0}$ on the point process, i.e.,
    \begin{equation}
        X_n\coloneqq \omega_{S_n},\qquad \forall~n\geq0.
    \end{equation}
    This process $(X_n)_{n\geq0}$ performs the same jumps as $(S_n)_{n\geq0}$ but on the points of $(\omega_r)_{r\in\mathbb{Z}}$, and this called the \textit{discrete-time} stochastic Lévy--Lorentz gas. Once we fix the realization of the discrete-time dynamics, we can use $(T^\omega_n)_{n\geq0}$ to denote the sequence of collision times by $T^\omega_0=0$ and
    \[
        T^\omega_n\coloneqq\sum_{k=1}^n |X_k-X_{k-1}| = \sum_{k=1}^n |\omega_{S_k}-\omega_{S_{k-1}}|,\qquad\forall~n\geq1.
    \]
    In rough terms, the length of the $n$th jump of the discrete-time stochastic Lévy--Lorentz gas is given by $|\omega_{S_k}-\omega_{S_{k-1}}|$, and whence $T^\omega_n$ represents the overall length of the trail up to time $n$.\par
    The motion of the \textit{continuous-time} stochastic Lévy--Lorentz gas is defined via piecewise interpolation of the sequence $(X_n,T^\omega_n)_{n\geq0}$ with the form
    \begin{equation}
        \widetilde{X}_t\coloneqq\omega_{S_{n-1}} + \text{sgn}(V_n)(t-T^\omega_{n-1}),\qquad\text{when}\quad T^\omega_{n-1}\leq t< T^\omega_n,\qquad\forall~t\geq0.
    \end{equation}
    One can see immediately that the continuous-time stochastic Lévy--Lorentz gas $(\widetilde{X}_t)_{t\geq0}$ normalizes the speed of the discrete-time process $(X_n)_{n\geq0}$. Remark that the displacement $\widetilde{X}_t$ is well-defined for all $t>0$, since $\lim_{n\to\infty} T^\omega_n=\infty$ in light of Lemma \ref{lem: T and n ratio limit} below.\par
    In sharp contrast to previous literature on the stochastic Lévy--Lorentz gas \cite{Barkai/Fleurov/Klafter,Zamparo}, we do not need the hypothesis that the distances $\zeta_r\coloneqq\omega_r-\omega_{r-1}$, $\omega_0=0$, $r\in\mathbb{Z}$ form an i.i.d. bilateral sequence. Instead, we recognize the \textbf{mutual interactions} among the scatterers $(\omega_r)_{r\in\mathbb{Z}}$ and such that the distances $(\zeta_r)_{r\in\mathbb{Z}}$ is considered to have stationary distribution but are allowed to entangle in the following sense,
    \begin{equation}\label{eqn: kkkkkkk}    
        \frac{d\mathbb{P}((\zeta_r)_{r\in\Delta}\in\vdot|\eta)}{d\mathbb{P}((\zeta_r)_{r\in\Delta}\in\vdot|\eta^\prime)}\leq \exp\bigg(C\sum_{r\in\Delta,u\in A}e^{-g\abs{r-u}}\bigg),\qquad\exists~C,g>0,
    \end{equation}
    simultaneously for all pairs of configurations $\eta,\eta^\prime\in\Omega$ which agree on $V^c\backslash A$, $\mathbb{P}$-a.s. Here we use $(\mathbb{P},\mathscr{F}_{\Omega})$ and $\mathbb{E}\coloneqq E_{\mathbb{P}}$ to denote the law of $(\zeta_r)_{r\in\mathbb{Z}}$. And $\Delta,A$ are arbitrary finite subsets of $\mathbb{Z}$ such that $\Delta\subseteq V\subseteq\mathbb{Z}$, $d_1(\Delta,V^c)\geq1$ and $A\subseteq V^c$. And we require {the regularity} $\mathbb{E}[\exp(\mathfrak{p}\zeta_1)]<\infty$ for some $\mathfrak{p}\geq2$, {which is employed in (\ref{eqn: A2 eqn}) of Lemma \ref{lem: omega LLN}}. { Notice the entanglement condition (\ref{eqn: kkkkkkk}) arises naturally from interacting statistical mechanics such as the Ising model, and can further be traced in spin chains of quantum Heisenberg model \cite{Heisenberg,Rojas} from the perspective of quantum computing theory \cite{Bravyi/Hastings,van den Nest/Dur/Briegel}. \textit{In summary}, the regularity conditions we impose on the non-i.i.d. scatterers $(\zeta_r)_{r\in\mathbb{Z}}$ are that $(\zeta_r)_{r\in\mathbb{Z}}$ have shift-invariant and ergodic distribution on $\mathbb{Z}$, satisfying the entanglement condition (\ref{eqn: kkkkkkk}) and having finite exponential moments $\mathbb{E}[\exp(\mathfrak{p}\zeta_1)]<\infty$ for some $\mathfrak{p}\geq2$.}\par
    The aim of this paper is two-fold: First we show the discrete-time stochastic Lévy--Lorentz gas $(X_n)_{n\geq0}$ and its continuous-time interpolation $(\widetilde{X}_t)_{t\geq0}$ both verify a central limit theorem (CLT) in the quenched scenario, i.e.~in a fixed realization of the scatterers $(\omega_r)_{r\in\mathbb{Z}}$.
    \begin{theorem}\label{thm: regular CLT}
        \normalfont
        We have the following convergences in distribution,
        \[
            \lim_{n\to\infty}\frac{1}{\sqrt{n}}\big(X_n - \ell E[S_1]n\big) \stackrel{\mathcal{L}}{=} \mathcal{N}(0,\ell^2 E[S_1^2])
        \]
        with the positive deterministic $\ell\coloneqq\lim_{n\to\infty}n^{-1}\omega_n$, $\mathbb{P}$-a.s., as well as the limiting identity for the continuous-time stochastic Lévy--Lorentz gas,
        \[
            \lim_{t\to\infty}\frac{1}{\sqrt{t}}\big(\widetilde{X}_t - \ell E[S_1]t\big) \stackrel{\mathcal{L}}{=} \mathcal{N}(0,\ell  E[S_1^2]/E[\abs{V_1}]),
        \]
        where in general $\mathcal{N}(0,\sigma^2)$ denotes the centered Gaussian variable with variance $\sigma^2>0$.
    \end{theorem}
    Second, we extend the above limit theorem to the more abstract Skorokhod space $\mathfrak{D}(\mathbb{R}_{+})$ of càdlàg processes on $[0,\infty)$, and thus obtain the functional CLT, or scaling limits, for both the discrete-time and continuous-time stochastic Lévy--Lorentz gases.
    \begin{theorem}\label{thm: functional CLT}
        \normalfont
        We have the following convergences
        \[
            \bigg(\frac{1}{\sqrt{n}} \big(X_{nt}- \ell E[S_1]nt \big),\; t\geq0\bigg)\Longrightarrow \bigg(W_t,\; t\geq0\bigg),
        \]
        where $(W_t)_{t\geq0}$ is a continuous real-valued centered Wiener process with covariance $E[W_sW_t]=\ell^2 E[S_1^2](s\wedge t)$ for all $0\leq s,t<\infty$, and where the convergence is respect to the weak topology of the Skorokhod space $\mathfrak{D}(\mathbb{R}_{+})$ of càdlàg processes. And we have as well the functional limit for the continuous-time model
        \[
            \bigg(\frac{1}{\sqrt{t}} \big(\widetilde{X}_{ts} - \ell E[S_1]ts \big),\; s\geq0\bigg)\Longrightarrow \bigg(\widetilde{W}_s,\; s\geq0\bigg),
        \]
        where $(\widetilde{W}_t)_{t\geq0}$ is a continuous real-valued centered Wiener process with covariance specified by $E[\widetilde{W}_s\widetilde{W}_t]=\ell E[S_1^2](s\wedge t)/E[\abs{V_1}]$ for all $0\leq s,t<\infty$,
        and where the convergence is respect to the weak topology on the distribution of finite-dimensional marginals.
    \end{theorem}
    The insights on scaling limits via finite-dimensional marginal distributions have also been discussed in \cite{Bianchi/Lenci/Pene} for i.i.d. scatterers. Due to the negative values of the particle $(S_n)_{n\geq0}$, the Skorokhod topology should be suitably extended to $\mathbb{R}$. Similar Skorokhod space extensions have been considered in \cite{Jacod/Shiryaev} and \cite{Magdziarz/Szczotka} where the scatterers $(\zeta)_{r\in\mathbb{Z}}$ are assumed i.i.d. In contrast, we give the scaling limits for both discrete and continuous-time gases in the more general interacting scatterers scenario with arbitrary Markovian random walk $(S_n)_{n\geq0}$.

\subsection{Step-reinforced Lévy--Lorentz gas.} 
    Motivated from the study of materials with memory \cite{Coleman}, we novelly take into consideration the memory of the particle traversing among the entangled scatterers. In light of such long-term memory effect, the underlying random walk is no longer Markovian and will now be denoted as $(S^{(p)}_n)_{n\geq0}$ with memory intensity $0\leq p\leq1$. The law $P^{(p)}_0$ and $E^{(p)}_0\coloneqq E_{P^{(p)}_0}$ of the random walk $(S^{(p)}_n)_{n\geq0}$ is formulated in the following sense.\par
    For the first step $V^{(p)}_1$, the reinforced random walk takes value in $\mathbb{Z}_2\coloneqq\{-1,+1\}$ with equal probability. The performance of the next steps are subject to their memories on previous steps. For each $(n+1)$th step, choose $k$ uniformly among the previous times $\{1,\ldots,n\}$, then the walk moves in the same direction as at time $k$ with probability $p$, or in the opposite direction with probability $1-p$. Namely, for all $n\geq1$,
    \begin{equation}
        S^{(p)}_{n+1} = S^{(p)}_n + V^{(p)}_{n+1},\qquad V^{(p)}_{n+1} = 
        \begin{cases}
            +V^{(p)}_k & \text{with probability}~p,\\
            -V^{(p)}_k & \text{with probability}~1-p,
        \end{cases}
    \end{equation}
    with $k\sim\mathcal{U}\{1,\ldots,n\}$. And thus $(S^{(p)}_n)_{n\geq0}$ is called the step-reinforced random walk.\par
    A particle with memory is a natural depiction of molecules which remember their past configurations \cite{Ortega-de San Luis/Ryan,Zeltser/Sukhanov/Nevorotin}. In the same weakly entangled scatterers $(\omega_r)_{r\in\mathbb{Z}}$, we can analogously define the step-reinforced Lévy--Lorentz gas
    \[
        X^{(p)}_n\coloneqq\omega_{S^{(p)}_n},\qquad \widetilde{X}^{(p)}_t\coloneqq\omega_{S^{(p)}_{n-1}}+ V^{(p)}_n(t-T^{(p)}_{n-1}),\qquad\text{when}\quad T^{(p)}_{n-1}\leq t< T^{(p)}_n,\qquad\forall~n\in\mathbb{N},
    \]
    with the sequence of step-reinforced collision times $T^{(p)}_n\coloneqq\sum_{k=1}^n|X^{(p)}_k-X^{(p)}_{k-1}|$ for each $n\in\mathbb{N}$. Remark that the continuous-time step-reinforced Lévy--Lorentz gas $(\widetilde{X}^{(p)}_t)_{t\geq0}$ is well-defined on $[0,\infty)$ since $T^{(p)}_n\to\infty$ as is shown in (\ref{eqn: 3.11}).\par
    Notice that at $p=1/2$ the step-reinforced random walk $(S^{(1/2)}_n)_{n\geq0}$ reduces to the nearest-neighbor symmetric random walk on $\mathbb{Z}$, whence the reinforced increments $(V^{(1/2)}_n)_{n\geq1}$ can be identified with an i.i.d. jump sequence on $\mathbb{Z}_2=\{-1,+1\}$, conforming to the memoryless model in Section \ref{sec: 2.1}. Incorporating the memory effects, when $0\leq p\leq 3/4$ we have the following limiting result.
    \begin{theorem}\label{thm: regular reinforced CLT}
        \normalfont
        When $0\leq p< 3/4$, we have the following convergences in distribution,
        \[
            \lim_{n\to\infty}\frac{1}{\sqrt{n}}X^{(p)}_n \stackrel{\mathcal{L}}{=} \mathcal{N}(0,\ell^2(3-4p)^{-1})
        \]
        with the positive deterministic $\ell\coloneqq\lim_{n\to\infty}n^{-1}\omega_n$, $\mathbb{P}$-a.s., and the limiting identity for the step-reinforced continuous-time stochastic Lévy--Lorentz gas,
        \[
            \lim_{t\to\infty}\frac{1}{\sqrt{t}} \widetilde{X}^{(p)}_t \stackrel{\mathcal{L}}{=} \mathcal{N}(0,\ell(3-4p)^{-1}),
        \]
        where in general $\mathcal{N}(0,\sigma^2)$ denotes the centered Gaussian variable with variance $\sigma^2>0$.
    \end{theorem}
    On the other hand, similar to the stochastic Lévy--Lorentz gas driven by Markovian random walk, the step-reinforced model also enjoys an analogous scaling limit, i.e.~functional CLT in the following sense.
    \begin{theorem}\label{thm: functional reinforced CLT}
        \normalfont
        When $0\leq p< 3/4$, we have the following convergences
        \[
            \bigg(\frac{1}{\sqrt{n}}X^{(p)}_{nt},\; t\geq0\bigg)\Longrightarrow \bigg(W^{(p)}_t,\; t\geq0\bigg),
        \]
        where $(W^{(p)}_t)_{t\geq0}$ is a continuous real-valued centered Wiener process with covariance $E[W^{(p)}_sW^{(p)}_t]=\ell^2 (3-4p)^{-1}(s\wedge t)$ for all $0\leq s,t<\infty$, and where the convergence is viewed from the weak topology of the Skorokhod space $\mathfrak{D}(\mathbb{R}_{+})$ of càdlàg processes. And we have as well the functional limit for the reinforced continuous-time model
        \[
            \bigg(\frac{1}{\sqrt{t}}\widetilde{X}^{(p)}_{ts},\; s\geq0\bigg)\Longrightarrow \bigg(\widetilde{W}^{(p)}_s,\; s\geq0\bigg),
        \]
        where $(\widetilde{W}^{(p)}_t)_{t\geq0}$ is a continuous real-valued centered Wiener process with the covariance specified by $E[\widetilde{W}^{(p)}_s\widetilde{W}^{(p)}_t]=\ell (3-4p)^{-1}(s\wedge t)$ for all $0\leq s,t<\infty$,
        and where the convergence is viewed from the weak topology on the distribution of finite-dimensional marginals.
    \end{theorem}
    {
    \begin{remark}
        It has now been stated the CLT for step-reinforced Lévy--Lorentz gas in the diffusive regime $0\leq p<3/4$. At the critical threshold $p=3/4$, it is shown \cite[Theorem 3.6]{Bercu} that $\frac{1}{\sqrt{n\log n}}S^{(3/4)}_n\Rightarrow\mathcal{N}(0,1)$ as $n\to\infty$ in distribution, and thus the corresponding CLT for Lévy--Lorentz gas $(X_n^{(3/4)})_{n\in\mathbb{N}}$ should also have a different scaling factor $\sqrt{n\log n}$ instead of $\sqrt{n}$. For the superdiffusive regime $3/4<p\leq1$, we only know $S^{(p)}_n/n^{2p-1}$ converges in $L^4(P^{(p)}_0)$-norm to a non-degenerate random variable denoted by $\mathcal{L}$ in \cite[Theorem 3.7]{Bercu}. And whence we will need new techniques to investigate the CLT for superdiffusive step-reinforced Lévy--Lorentz gas, which is beyond the scope of this paper.
    \end{remark}
    }
    {
    \begin{remark}
        Here we emphasize that the stated convergences from Theorems \ref{thm: regular CLT},\ref{thm: functional CLT},\ref{thm: regular reinforced CLT},\ref{thm: functional reinforced CLT} hold for the quenched law of the process $\mathbb{P}$-a.s. with respect to the underlying environment. Further, the value $\ell=\lim_{n\to\infty}n^{-1}\omega_n$, apart from being defined as a $\mathbb{P}$-a.s. limit, can also be written as an explicit average of a function of the interdistance $\zeta_r$'s. One interesting observation is that $\ell$ reduces to $\mathbb{E}[\zeta_r]$, $\forall~r\in\mathbb{Z}$ for i.i.d. scatterers by the law of large numbers.
    \end{remark}
    }
    The novelty of considering reinforced particles in the study of stochastic Lévy--Lorentz gas sheds light on other delicate limiting structure of such model, for instance the fluctuation \cite{Zamparo} and large deviations \cite{Chen0,Chen}. The memory effect will be investigated via a martingale approach in Section \ref{sec: 3.3}. {The mathematical framework is not easy at first reading. But the arguments can more or less be categorized to two techniques. With the correlated coordinate scatterers $(\zeta_r)_{r\in\mathbb{Z}}$, we employ an auxiliary random walk to manually separate the mixing information from the entangled integer lattice and force its limiting distribution to have the same law of large numbers as the original underlying random walk. Second, going from the LLN to central limit theorem requires a Portemanteau-type argumentfor the non-reinforced Lévy--Lorentz gas and a martingale approach to dissemble the memory of the past for the step-reinforced model. These techniques consist of the technical arguments in the following sections of this paper.} Recently, other technical tools such as fixed-point approach \cite{Guerin/Laulin/Raschel} and Pólya-urn embedding \cite{Laulin} may also be useful for revealing other statistics of step-reinforced Lévy--Lorentz gases.

\section{Regular central limit theorem of Lévy--Lorentz gas}
    In this section, we prove the CLT, i.e.~Theorem \ref{thm: regular CLT} in both the discrete-time and continuous-time models.
\subsection{Discrete-time scenario.} We will need the following technical lemmas. Recall the splitting representation: If $\Bar{X},\Tilde{X}$ are random variables of law $\Bar{P},\Tilde{P}$ with $\normx{\Bar{P}-\Tilde{P}}_{FV}\leq a<1$, then on an enlarged probability space there exists independent $Y,\delta,Z,\Tilde{Z}$, where $\Delta\sim\text{Bernoulli}(a)$ and
    \begin{equation}\label{eqn: A eqn1}            
        \Bar{X}=(1-\Delta)Y+\Delta Z\qquad\text{and}\qquad\Tilde{X}=(1-\Delta)Y+\Delta\Tilde{Z}.
    \end{equation}
    For a proof, see \cite[Appendix A.1]{Barbour/Holst/Janson}. Although, the exact form of $Y,Z$ are complicated, we nevertheless have the estimates $\Bar{X}=(1-\Delta)\Tilde{X}+\Delta Z$, $|\Delta Z|\leq|\Bar{X}|$ and $|\Delta\Tilde{Z}|\leq|\Tilde{X}|$. Note that by recursive conditioning, this result extends to random sequences.
    \begin{lemma}\label{lem: A1}
        \normalfont
        Given a random sequence $(X_i)_{i\geq1}$ with law $P$ such that for some probability measure $Q$,
        \[
            \normy{P(\Bar{X}\in\vdot|\Bar{X}_j,\,j<i)-Q }_{FV}\leq a <1.
        \]
        Then there exists an i.i.d sequence $(\Tilde{X}_i,\Delta_i)_{i\geq1}$ such that $\Tilde{X}_1\sim Q$ and $\Delta_1\sim\text{Bernoulli}(a)$ on an enlarged probability space, as well as a sequence $(Z_i)_{i\geq1}$ with $\Delta_i$ independent of $\sigma(\Tilde{X}_j,\Delta_j:\,j<i)\vee\sigma(Z_i)$ and
        \[
            \Tilde{X}_i=(1-\Delta_i)\Tilde{X}_i + \Delta_iZ_i,\qquad\forall~i\geq1.
        \]
    \end{lemma}
    \begin{proof}
        Suppose the assertion has been verified for $i-1$. Apply (\ref{eqn: A eqn1}) with laws $P(\Bar{X}_i\in\vdot|\Bar{X}_j,\,j<i)$ and $Q$, we obtain of $\Delta_i,Z_i,\Tilde{X}_i$ with desired properties. See also \cite[Lemma 2.1]{Berbee} and \cite[Chapter 3]{Thorisson}.
    \end{proof}
    To resolve the weakly interacting distances array $(\zeta_r)_{r\in\mathbb{Z}}$ of the scatterers, we use an approximate renewal structure to separate the entangled information from the disordered medium. To this end, we create an auxiliary probability space. Define $\mathcal{W}\coloneqq\{-1,0,1\}$ and define the product probability measure $Q$ on $\epsilon=(\epsilon_1,\epsilon_2,\ldots)\in(\mathcal{W})^{\mathbb{N}}$ by $Q(\epsilon_1=1)=Q(\epsilon_1=-1)=\frac{1}{4}$, $Q(\epsilon_1=0)=\frac{1}{2}$. We also introduce the auxiliary random field $\eta=(\eta_{r})_{r\in\mathbb{N}}$ by
    \begin{equation}\label{eqn: needed}    
        \tfrac{1}{2}\eta_{r} = E_Q[\zeta_{r}]\mathbbm{1}_{\{\epsilon_r=\pm1\}} + (\zeta_{r}-E_Q[\zeta_{r}]) \mathbbm{1}_{\{\epsilon_r=0\}},\qquad\forall~r\in\mathbb{N}
    \end{equation}
    It is not hard to see that $E_Q[\eta_{r}]=\zeta_{r} $ for all $r\in\mathbb{N}$. Following this convention, we denote $\overline{P}^S_0\coloneqq Q\otimes P^S_0$ and subsequently $\overline{E}^S_0\coloneqq E_{Q\otimes P^S_0}$. Let us now select a random time sequence $(\tau^{(L)}_n)_{n\geq0}$ given by $\tau^{(L)}_0=0$ and
    \[
        \tau^{(L)}_n\coloneqq\inf\big\{j\geq\tau^{(L)}_{n-1}+L:\,(\epsilon_{j-L},\ldots,\epsilon_{j-1})=+1,\ldots,1,\,\epsilon_j=-1,0\big\},\qquad\forall~n\geq1.
    \]
    Here $L$ is the length of the large distance with which we separate the entangled information.\par
    To quantify the correlated information from the interacting scatterers, we therefore also define the following random field $\sigma$-algebras on $\mathbb{N}\times(\mathcal{W})^{\mathbb{N}}$ by $\mathscr{G}_0\coloneqq\sigma(\zeta_{r}:\,r\leq-L)$, and 
    \[
        \mathscr{G}_n\coloneqq\sigma\big(\tau^{(L)}_1,\ldots,\tau^{(L)}_n,\,\zeta_{r}:\,r\leq\tau^{(L)}_n-L,\,\epsilon_i:\,i=1,\ldots,\tau^{(L)}_n\big),\qquad\forall~n\geq1.
    \]
    We also define the specific field $\sigma$-algebras $\mathscr{F}^{(L)}_n\coloneqq\sigma(\omega_{r}:\,r\leq n-L,\,\epsilon_i:\,i=1,\ldots,n)$. And in fact we have the following regularity proposition.
    \begin{lemma}\label{lem: tau lower bound}
        \normalfont
        For any real $p\geq1$, there exists some constants $c_p,c^\prime_p>0$ independent of $L$ such that
        \[
            c_p\leq E_Q\big[(\Bar{\tau}^{(L)}_1)^p\big]^{1/p}\leq c^\prime_p
        \]
        for all $L$, where we let $\Bar{\tau}^{(L)}_1\coloneqq4^{-L}\tau^{(L)}_1$.
    \end{lemma}
    \begin{proof}
        We divide the proof of this lemma into several steps.\par\noindent
        \textbf{Step I.} Lower-bound.\par\noindent
        Define a Markov Chain $(U_n)_{n\geq1}$ on state space $\{0,1,\ldots,L\}$, with $U_0=0$ and 
        \[
            U_n=\max\big\{k\geq1:\,(\epsilon_{n-k+1},\ldots,\epsilon_n)=(1,\ldots,1)\big\}\vee0.
        \]
        Then, it is not hard to see that $\tau^{(L)}_1\geq\min\{n\geq1:\,U_n=L\}$, see \cite[p. 892]{Comets/Zeitouni}. Consider the successive times when $U_n=1$, then $\tau^{(L)}_1$ can be lower-bounded by a sum of a $\text{Geometric}(4^{-L+1})$ number of independent random variables which are bounded below by $1$. Thus,
        \[
            \varliminf_{L\to\infty} Q\big(\tau^{(L)}_1\geq\theta4^{-L}\big)\geq g(\theta)\qquad\text{for some}\quad g(\theta)\xrightarrow[]{\theta\to0}1.
        \]
        Therefore,
        \[
            E_Q\big[\Bar{\tau}^{(L)}_1\big]\geq 4^{-L}E_Q\big[\tau^{(L)}_1\mathbbm{1}_{\{ \tau^{(L)}_1\geq\theta4^{-L} \}}\big]\geq \theta\big(1-Q(\tau^{(L)}_1<\theta4^{-L})\big)\geq c>0,
        \]
        provided that we choose $\theta$ sufficiently small such that $\varlimsup_{L\to\infty} Q(\tau^{(L)}_1<\theta4^{-L})<\frac{1}{2}$.\par\noindent
        \textbf{Step II.} Upper-bound.\par\noindent
        We will actually prove the exponential moment of $\Bar{\tau}^{(L)}_1$ is finite, which is nevertheless stronger then the claim. Define the events for each $n\geq1$ by
        \[
        A_n\coloneqq\big\{\epsilon\in(\mathcal{W})^{\mathbb{N}}:\,(\epsilon_{n-L},\ldots,\epsilon_{n-1})=1,\ldots,1,\,\epsilon_n=-1,0\big\}
        \]
        and 
        \[
            B_n\coloneqq\big\{\epsilon\in(\mathcal{W})^{\mathbb{N}}:\,(\epsilon_{j-L},\ldots,\epsilon_{j-1},\epsilon_j)\neq1,\ldots,1,-1\,\text{or}\,0,\,\forall~L\leq j\leq n-L-1\big\}.
        \]
        Hence,
        \begin{equation*}\begin{aligned}
            E_Q\big[e^{\Bar{\tau}^{(L)}_1}\big]/3 &\leq \sum_{n=1}^{2L} E_Q\big[e^{4^{-L}n},\,A_n\big] +\sum_{n=2L+1}^\infty E_Q\big[e^{4^{-L}n},\,A_n,\,B_n\big]\\ 
            &\leq (L^2-1)4^{-(L+1)} e^{L4^{-(L-1)}} + 4^{-(L+1)} \sum_{k=L^2}^\infty E_Q\big[e^{4^{-L}n},\,B_n\big].
        \end{aligned}\end{equation*}
        By \cite[Lemma 6.6]{Guerra Aguilar/Ramirez}, $Q(B_n)\leq (1-cL^2 4^{-L})^{ \lfloor n/L^2\rfloor }$ for each $n\geq L^2$. Hence,
        \[
            E_Q\big[e^{\Bar{\tau}^{(L)}_1}\big] \leq K + \tfrac{3}{2}L^2 4^{-L} \sum_{k=1}^\infty \big(e^{L^2 4^{-L}}(1-cL^2 4^{-L})\big)^k\leq K+ \tfrac{3 K^2 4^{-L}/2}{e^{-L^2 4^{-L}}- (1-cL^2 4^{-L}) } \leq K+K^\prime<\infty,
        \]
        where $K,K^\prime$ are absolute constants independent of $L$. And the assertion is verified.   
    \end{proof}
    \begin{lemma}\label{lem: tau mixing inequality}
        \normalfont
        Let $f:(\mathbb{R})^{\mathbb{N}}\to\mathbb{R}$ be a bounded Borel measurable function. For any $n\in\mathbb{N}$, we abbreviate both the finite-time process $(\sum_{r=\tau^{(L)}_{n-1}+1}^i \eta_{r})_{\tau^{(L)}_{n-1}\leq i\leq\tau^{(L)}_n}$ and the process $(\sum_{r=\tau^{(L)}_{n-1}+1}^i \zeta_{r})_{\tau^{(L)}_{n-1}\leq i\leq\tau^{(L)}_n}$ by $S^{\tau,n}_{\vdot}$. Then, under the weakly entangled condition and for each $t\in(0,1)$ there exists $L_0 = L_0(C, g, t) \in \mathbb{N}$ so that $\mathbb{P}\otimes P^S_0\otimes Q$-a.s. for all $L\geq L_0$,
        \begin{equation*}\begin{aligned}                                        
            \exp \big(-e^{-gtL}\big) \mathbb{E}\overline{E}^S_0 \big[f(S^{\tau,1}_{\vdot})\big] \leq \mathbb{E}\overline{E}^S_0 \big[f(S^{\tau,n}_{\vdot}) \big|\mathscr{G}_{n-1}\big]\leq \exp\big(e^{-gtL}\big)\mathbb{E}\overline{E}^S_0\big[f(S^{\tau,1}_{\vdot})\big].
        \end{aligned}\end{equation*}
    \end{lemma}
    \begin{proof}
        Take bounded and $\mathscr{G}_{n-1}$-measurable $h:(\mathbb{R})^{\mathbb{N}\times\mathbb{Z}^d}\times(\mathcal{W})^{\mathbb{N}}\to\mathbb{R}$, where $\mathcal{W}\coloneqq\{-1,0,1\}$. Then,
        \[
            \mathbb{E}\overline{E}^S_0 \big[f(S^{\tau,n}_{\vdot})h\big] = \sum_{k\in\mathbb{N}} \mathbb{E}\overline{E}^S_0 \big[f(S^{\tau,n}_{\vdot}) h_k,\,\tau^{(L)}_{n-1}=k\big],
        \]
        because on the event $\{\tau^{(L)}_{n-1}=k\}$, we can find a bounded function $h_{k}$ which is $\mathscr{F}^{(L)}_{k}$-measurable and coincides with $h$ on this event. Then, we observe that
        \begin{equation}\label{eqn: decompose conditional expectation}
            \mathbb{E}\overline{E}^S_0\big[f(S^{\tau,n}_{\vdot}) h\big] = \sum_{k\in\mathbb{N}} \mathbb{E}\overline{E}^S_0 \big[ h_k,\, \tau^{(L)}_{n-1}=k,\, \mathbb{E}\overline{E}^S_0 [f(S^{\tau,1}_{k+\vdot})|\mathscr{F}^{(L)}_{k}]\big],
        \end{equation}
        Consider the space-time hyperplane $\mathbb{H}_{L,k}$ and the one-dimensional cone-like region $\mathbb{C}_{k}$ defined resp. by
        \[
            \mathbb{H}_{L,k}\coloneqq\big\{m\in \mathbb{Z}:\, m-k\leq-L \big\} \qquad\text{and}\qquad\mathbb{C}_{k}\coloneqq \mathbb{Z}_{\geq k}.
        \]
        In terms of the weakly interacting condition on scatterers, we estimate the series
        \begin{equation}\label{eqn: estimate series2, SMG}
            \sum_{\Vec{y}\in\mathbb{H}_{L,k,x}} \sum_{\Vec{z}\in\mathbb{C}_{k,x}}\exp\big(-g\abs{\Vec{y}-\Vec{z}}_1\big).
        \end{equation}
        Notice that with $L$ sufficiently large, this series converge because $\mathbb{C}_{k}$ is cone-like. Indeed, choose $\hat{t}\in(t,1)$ and consider (\ref{eqn: estimate series2, SMG}). We take $L$ large enough such that $L>(1-\hat{t})^{-1}2r$, and thus $L-2r>\hat{t}L$. Setting
        \[
            \mathbb{K}_{L,k,n}\coloneqq\big\{(\Vec{y},\Vec{z}):\,\Vec{y}\in \mathbb{H}_{L,k},\,\Vec{z}\in \mathbb{C}_{k},\,\hat{t}L+n\leq\abs{\Vec{y}-\Vec{z}}_1<\hat{t}L+n+1\big\}.
        \]
        Whence we have
        \[
            \sum_{\Vec{y}\in\mathbb{H}_{L,k}} \sum_{\Vec{z}\in\mathbb{C}_{k}}\exp\big(-g\abs{y-z}_1\big) \leq \sum_{n\geq0}\sum_{(\Vec{y},\Vec{z})\in\mathbb{K}_{L,k,n}} e^{-g\abs{\Vec{y}-\Vec{z}}_1} \leq \sum_{n\geq0} \abs{\mathbb{K}_{L,k,n}}e^{-g(\hat{t}L+n)}.
        \]
        Since $\abs{\mathbb{K}_{L,k,n}}\leq Cr^2(n+1)$, we have
        \begin{equation}\label{eqn: aaaa}
            \sum_{\Vec{y}\in\mathbb{H}_{L,k}} \sum_{\Vec{z}\in\mathbb{C}_{k}}\exp\big(-g\abs{\Vec{y}-\Vec{z}}_1\big) \leq \sum_{n\geq0} (n+1) e^{-gn} \leq e^{-g\Tilde{t}L},\quad\text{with}~\Tilde{t}\in(t,\hat{t}),
        \end{equation}
        yielding an estimate to (\ref{eqn: estimate series2, SMG}). From the last term in (\ref{eqn: decompose conditional expectation}), we have
        \begin{equation*}
            \mathbb{E}\overline{E}^S_0 \big[f(S^{\tau,1}_{k+\vdot})\big|\mathscr{F}^{(L)}_{k}\big] = \mathbb{E}\overline{E}^S_0\big[f(S^{\tau,1}_{k+\vdot}) \big|\mathscr{F}_{\mathbb{H}_{L,\hat{\gamma}k}}\big],
        \end{equation*}
        because $Q$ is a product probability measure on $(\mathcal{W})^{\mathbb{N}}$. We further denote by 
        \[
            \mathbb{H}^{(n)}_{L,k}\coloneqq\mathbb{H}_{L,k}\cup\{\Vec{z}\in\mathbb{Z}:\,\abs{\Vec{z}}_1\geq n\},\quad\forall~k\in\mathbb{Z}.
        \]
        The weakly entangled condition on the scatterers together with (\ref{eqn: aaaa}) thus implies that
        \[
            \exp\big(-Ce^{-g\Tilde{t}L}\big) \leq \frac{ \mathbb{E}\overline{E}^S_0\big[f(S^{\tau,1}_{k+\vdot}),\,\gamma = (S_{k+i}-S_k)_{0\leq i\leq \tau^{(L)}_1} \big|\mathscr{F}_{\mathbb{H}^{(n)}_{L,\hat{\gamma}k}}\big] }{\mathbb{E} \overline{E}^S_0\big[f(S^{\tau,1}_{k+\vdot}),\,\gamma = (S_{k+i}-S_k)_{0\leq i\leq \tau^{(L)}_1} \big]}\leq \exp\big(Ce^{-g\Tilde{t}L}\big)
        \]
        uniformly for bounded $f$ and for any finite path $\gamma$ satisfying (\ref{eqn: aaaa}). Taking some suitable $\Bar{t}\in(t,\Tilde{t})$ and letting $n\to\infty$, we have
        \begin{equation}\label{eqn: ghi}                   
            \exp\big(-e^{-g\Bar{t}L}\big) \leq \frac{\sum_\gamma \mathbb{E}\overline{E}^S_0\big[f(S^{\tau,1}_{k+\vdot}),\,\gamma = (S_{k+i}-S_k)_{0\leq i\leq \tau^{(L)}_1} \big|\mathscr{F}_{\mathbb{H}^{(n)}_{L,\hat{\gamma}k}}\big] }{\sum_\gamma \mathbb{E} \overline{E}^S_0\big[f(S^{\tau,1}_{k+\vdot}),\,\gamma = (S_{k+i}-S_k)_{0\leq i\leq \tau^{(L)}_1} \big]} \leq \exp\big(e^{-g\Bar{t}L}\big).
        \end{equation}
        Taking (\ref{eqn: ghi}) into (\ref{eqn: decompose conditional expectation}), letting $L_0(t)$ be sufficiently large and $L\geq L_0$, we verify the assertion.       
    \end{proof}
    \begin{lemma}\label{lem: omega LLN}
        \normalfont
        There exists a deterministic limit $\ell>0$ such that
        \[
            \lim_{n\to\infty} \tfrac{1}{n}\omega_{\alpha n+\beta} = \lim_{n\to\infty}\tfrac{1}{n} = \sum_{r\leq\alpha n+\beta} \eta_r = \alpha\ell,\qquad\mathbb{P}\otimes Q\text{-a.s.,}\qquad\forall~\alpha\in\mathbb{R}\quad\text{and}\quad\beta\in\mathbb{Z},
        \]
        where $\omega_{\alpha n+\beta}$ is the abbreviation of $\omega_{[\alpha n+\beta]}$ without any ambiguity.
    \end{lemma}
    Notice that we have only defined the auxiliary field $(\eta_r)_{r\in\mathbb{N}}$ on $\mathbb{N}$. Nevertheless, its definition on $\mathbb{Z}_{<0}$ is analogous to (\ref{eqn: needed}), and we omit the details here. And without loss of any generality, we take $\alpha=1$, $\beta=0$.
    \begin{proof}
        For each $i\geq1$, let $\Bar{X}^{(L)}_i\coloneqq \sum_{k=\tau^{(L)}_{i-1}+1}^{\tau^{(L)}_i} 4^{-L}\eta_{k}$ and $\Bar{\tau}^{(L)}_i\coloneqq 4^{-L}(\tau^{(L)}_i-\tau^{(L)}_{i-1})$, and let $\mu^{(L)}(\vdot)$ denote the law of $\Bar{X}^{(L)}_1$. By Lemma \ref{lem: tau mixing inequality}, it is easily seen for any $k\geq2$ that
        \[
            \abs{ \mathbb{E}\overline{E}^S_0\big[ \Bar{X}^{(L)}_k\in A\big|\mathscr{G}_{k-1} \big] - \mu^{(L)}(A)  }\leq\psi_L<1,\qquad\forall~\text{Borel}~A\subseteq\mathbb{R},
        \]
        where $\psi_L\coloneqq\exp(e^{-gtL})-1$. Taking supremum over all Borel sets $A\subseteq\mathbb{R}$, we get
        \[
            \normy{ \mathbb{E}\overline{E}^S_0\big[\Bar{X}^{(L)}_k\in\vdot\big|\mathscr{G}_{k-1}\big] - \mu^{(L)}(\vdot) }_{FV} \leq \psi_L,\qquad\forall~k\geq2.
        \]
        Invoking Lemma \ref{lem: A1}, we find the i.i.d. sequence $(\Tilde{X}^{(L)}_i,\Delta^{(L)}_i)_{i\geq1}$ so that $\Tilde{X}^{(L)}_1\sim\mu^{(L)}$, $\Delta^{(L)}_1\sim\text{Bernoulli}(\psi_L)$, along with the other sequence $(Z^{(L)}_i)_{i\geq1}$ satisfying
        \[
            \Bar{X}^{(L)}_i=(1-\Delta^{(L)}_i)\Tilde{X}^{(L)}_i+\Delta^{(L)}_i Z^{(L)}_i.
        \]
        We also write the enlarged $\sigma$-algebra $\Tilde{\mathscr{G}}_i\coloneqq\mathscr{G}_i\vee\sigma(\Tilde{X}^{(L)}_j,\Delta^{(L)}_j:\,j\leq i)$. Notice that we also have $|\Delta^{(L)}_i Z^{(L)}_i|\leq |\Bar{X}^{(L)}_i|$ for each $i\geq1$. Henceforth, by Hölder's inequality and for all real and even $1<p\leq\mathfrak{p}$,
        \begin{equation}\label{eqn: A2 eqn}
            \psi_L\mathbb{E}\overline{E}^S_0\big[(Z^{(L)}_i)^p\big|\Tilde{\mathscr{G}}_{i-1}\big] \leq \mathbb{E}\overline{E}^S_0\big[(\Delta^{(L)}_i Z^{(L)}_i)^p\big|\Tilde{\mathscr{G}}_{i-1}\big] \leq 2K(p) 2^{-pL }\exp(e^{-gtL}) E_Q\big[(\tau^{(L)}_1)^p\big],\;\;\mathbb{P}\otimes P^S_0\otimes Q\text{-a.s.,}
        \end{equation}
        where $K(p)\coloneqq \mathbb{E}[e^{p\zeta_1}]+1<\infty$ and we have used the fact that
        \begin{equation*}\begin{aligned}
            &\mathbb{E}\overline{E}^S_0\big[(\Bar{X}^{(L)}_i)^p\big] \leq \mathbb{E}\overline{E}^S_0\bigg[ \big( \sum_{k=\tau^{(L)}_{i-1}+1}^{\tau^{(L)}_i} e^{\eta_{k}} \big)^p \bigg]\\ 
            &\quad \leq 2^{-pL} E_Q\bigg[ (\tau^{(L)}_{i}-\tau^{(L)}_{i-1})^{p-1}  \sum_{k=\tau^{(L)}_{i-1}+1}^{\tau^{(L)}_i} \mathbb{E}E^S_0\big[ e^{p\eta_{k}}\big| \sigma(\tau^{(L)}_j:\,j\leq i) \big] \bigg]  \leq 2E_Q\big[(\Bar{\tau}^{(L)}_1)^p\big] K(p).
        \end{aligned}\end{equation*}
        Notice that $\varlimsup_{L\to\infty} E_Q[(\Bar{\tau}^{(L)}_1)^p]<\infty$ by Lemma \ref{lem: tau lower bound}. We can express
        \begin{equation}\label{eqn: log need 1}        
            \frac{1}{n}\sum_{i=1}^n\Bar{X}^{(L)}_i=\frac{1}{n}\sum_{i=1}^n \Tilde{X}^{(L)}_i - \frac{1}{n}\sum_{i=1}^n\Delta_i^{(L)}\Tilde{X}^{(L)}_i + \frac{1}{n}\sum_{i=1}^n\Delta^{(L)}_i Z^{(L)}_i,
        \end{equation}
        where first by independence 
        \[
            \frac{1}{n}\sum_{i=1}^n\Tilde{X}^{(L)}_i\xrightarrow[]{n\to\infty}\gamma_L,\qquad\text{where}\quad\gamma_L\coloneqq \mathbb{E}\overline{E}^S_0[\Tilde{X}^{(L)}_1],\qquad\mathbb{P}\otimes P^S_0\otimes Q\text{-a.s.}
        \]
        Meanwhile, for any conjugate $p,q>1$ with $\frac{1}{p}+\frac{1}{q}=1$,
        \[
            \varlimsup_{n\to\infty}\bigg|\frac{1}{n}\sum_{i=1}^n \Delta^{(L)}_i\Tilde{X}^{(L)}_i\bigg|\leq\varlimsup_{n\to\infty} \bigg|\frac{1}{n}\sum_{i=1}^n (\Delta^{(L)}_i)^p \bigg|^{1/p}\vdot \bigg| \frac{1}{n}\sum_{i=1}^n(\Tilde{X}^{(L)}_i)^q \bigg|^{1/q}\leq 2K(q)\psi_L^{1/p} E_Q\big[ (\Bar{\tau}^{(L)}_1)^q \big]^{1/q},
        \]
        $\mathbb{P}\otimes P^S_0\otimes Q$-a.s., where $\eta_L\xrightarrow[]{L}0$ and the last inequality is due to (\ref{eqn: A2 eqn}). Let us define $\Bar{Z}^{(L)}_i\coloneqq \mathbb{E}\overline{E}^S_0[Z^{(L)}_i|\Tilde{\mathscr{G}}_{i-1}]$ for each $i\geq1$. Observe that the process $(M^{(L)}_n)_{n\geq1}$ with each $M^{(L)}_n\coloneqq \sum_{i=1}^n i^{-1}\Delta^{(L)}_i (Z^{(L)}_i-\Bar{Z}^{(L)}_i) $ is a centered $(\Tilde{\mathscr{G}}_n)_{n\geq1}$-martingale. By the Burkholder--Gundy maximal inequality \cite[Eqn. (14.18)] {Zerner/Merkl},
        \begin{equation*}\begin{aligned}
            &\mathbb{E}\overline{E}^S_0\bigg[\big|\sup_{n\geq1} M^{(L)}_n \big|^\gamma\bigg] \leq C(\gamma) \mathbb{E}\overline{E}^S_0\bigg[\sum_{n=1}^\infty \frac{1}{n^2}\big(\Delta^{(L)}_n Z^{(L)}_n - \Delta^{(L)}_n \Bar{Z}^{(L)}_n \big)^2 \bigg]^{\gamma/2}\\ 
            &\qquad\leq C(\gamma)\sum_{n=1}^\infty\frac{1}{n^\gamma} \mathbb{E}\overline{E}^S_0\big[\big(\Delta^{(L)}_n Z^{(L)}_n - \Delta^{(L)}_n \Bar{Z}^{(L)}_n \big)^\gamma\big] \leq C^\prime(\gamma)\qquad\text{with}\quad\gamma\coloneqq p\wedge q.
        \end{aligned}\end{equation*}
        Henceforth, $M^{(L)}_n\xrightarrow[]{n} M^{(L)}_\infty$, $\mathbb{P}\otimes P^S_0\otimes Q$-a.s. with integrable limit $M^{(L)}_\infty$. By Kronecker's lemma \cite[Eqn. (12.7)]{Zerner/Merkl}, it follows that $n^{-1}\sum_{i=1}^n\Delta^{(L)}_i(Z^{(L)}_i -  \Bar{Z}^{(L)}_i)\xrightarrow[]{n}0$, $\mathbb{P}\otimes P^S_0\otimes Q$-a.s. Thus with real and even $q>1$, for any $n\geq1$, by (\ref{eqn: A2 eqn}),
        \[
            |\Bar{Z}^{(L)}_n|\leq \mathbb{E}\overline{E}^S_0\big[ (\Bar{Z}^{(L)}_n)^q\big|\Tilde{\mathscr{G}}_{n-1} \big]^{1/q} \leq K(q)\exp(q^{-1}e^{-gtL})\normx{\Bar{\tau}^{(L)}_1}_{L^q(Q)}\psi_L^{-1/q}\eqqcolon\eta_L\psi_L^{-1/q},.
        \]
        Henceforth by independence, 
        \[
            \varlimsup_{n\to\infty} \bigg|\frac{1}{n}\sum_{i=1}^n\Bar{Z}^{(L)}_i\Delta^{(L)}_i\bigg|\leq \varlimsup_{n\to\infty} \eta_L\psi_L^{-1/q}\frac{1}{n}\sum_{i=1}^n\Delta^{(L)}_i\leq \eta_L\psi_L^{1/p},\qquad\text{where}\quad\frac{1}{p}+\frac{1}{q}=1.
        \]
        Combining all the above estimates, we get
        \begin{equation}\label{eqn: A3 eqn}
            \varlimsup_{n\to\infty}\bigg|\frac{1}{n}\sum_{i=1}^n\Bar{X}^{(L)}_i-\gamma_L\bigg|\leq2\eta_L\psi_L^{1/p}\xrightarrow[]{L\to\infty}0,\qquad\mathbb{P}\otimes P^S_0\otimes Q\text{-a.s.}
        \end{equation}
        On the other hand, it is immediate that
        \[
            \frac{1}{n}\sum_{i=1}^n \Bar{\tau}^{(L)}_i\xrightarrow[]{n\to\infty} \beta_L\coloneqq E_Q\big[\Bar{\tau}^{(L)}_1\big],\qquad Q\text{-a.s.,}
        \]
        by independence. Furthermore, by Lemma \ref{lem: tau lower bound}, $E_Q[\Bar{\tau}^{(L)}_1]\geq c>0$. Therefore, together with (\ref{eqn: A3 eqn}),
        \begin{equation}\label{eqn: A4}
            \varlimsup_{n\to\infty}\bigg| \frac{1}{\tau_n^{(L)}}\sum_{k=1}^{\tau_n^{(L)}}\eta_{k}-\frac{\gamma_L}{\beta_L} \bigg|\leq\varlimsup_{n\to\infty}\bigg| \frac{n^{-1}\sum_{i=1}^n\Bar{X}^{(L)}_i}{n^{-1}\sum_{i=1}^n\Bar{\tau}^{(L)}_i} - \frac{\gamma_L}{\beta_L} \bigg| \leq C\eta_L\psi_L^{1/p},\qquad\mathbb{P}\otimes P^S_0\otimes Q\text{-a.s.}
        \end{equation}
        Following standard arguments \cite[p. 1864]{Sznitman/Zerner}, we define an increasing sequence $(k_n)_{n\geq1}$ satisfying $\tau^{(L)}_{k_n}\leq n< \tau^{(L)}_{k_n+1}$ for all $n$. Then, we can write
        \[
            \frac{1}{n}\sum_{k=1}^n\eta_{k} = \frac{k_n}{n}\vdot \frac{1}{k_n}\bigg(\sum_{k=1}^{\tau^{(L)}_{k_n}} \eta_{k} +\sum_{k=\tau^{(L)}_{k_n}+1}^n \eta_{k} \bigg).
        \]
        It is already clear that 
        \[
        \frac{k_n}{n}\xrightarrow[]{n\to\infty}4^{-L}\frac{1}{E_Q[\Bar{\tau}^{(L)}_1]},\qquad Q\text{-a.s.}
        \]
        and that
        \[
            \gamma_L-2\eta_L\psi_L^{1/p}\leq 4^{-L}\varliminf_{n\to\infty} \frac{1}{k_n} \sum_{k=1}^{\tau^{(L)}_{k_n}} \eta_{k} \leq 4^{-L}\varlimsup_{n\to\infty} \frac{1}{k_n} \sum_{k=1}^{\tau^{(L)}_{k_n}} \eta_{k} \leq \gamma_L+2\eta_L\psi_L^{1/p},\qquad\mathbb{P}\otimes P^S_0\otimes Q\text{-a.s.}
        \]
        Furthermore, let us define $\Bar{\gamma}_L\coloneqq 4^{-L} \mathbb{E}\overline{E}^S_0[\sum_{k=1}^{\tau^{(L)}_1}\abs{\eta_{k}}]$. Following an almost identical argument to the above computations, we have
        \begin{equation}\label{eqn: log need 2}        
            4^{-L}\varlimsup_{n\to\infty}\frac{1}{k_n}\sum_{k=\tau^{(L)}_{k_n}+1}^n \eta_{k} \leq 4^{-L} \varlimsup_{n\to\infty}\frac{1}{k_n} \bigg( \sum_{k=1}^{\tau^{(L)}_{k_n+1}}  - \sum_{k=1}^{\tau^{(L)}_{k_n}} \bigg)\abs{\eta_{k}} \leq \Bar{\gamma}_L + 2\eta_L\psi_L^{1/p} - (\Bar{\gamma}_L-2\eta_L\psi_L^{1/p}).
        \end{equation}
        And similarly, for the lower-bound we have
        \[
            4^{-L}\varliminf_{n\to\infty}\frac{1}{k_n}\sum_{k=\tau^{(L)}_{k_n+1}}^n \zeta_{k} \geq -\Bar{\gamma}_L - 2\eta_L\psi_L^{1/p} + (\Bar{\gamma}_L+2\eta_L\psi_L^{1/p})\xrightarrow[]{L\to\infty}0,\qquad\mathbb{P}\otimes P^S_0\otimes Q\text{-a.s.}
        \]
        Combining the above estimates, we get the desired law of large numbers for $N^{-1}\sum_{k=1}^N\eta_{k}$ with the limit $\ell=\lim_{L\to\infty}\gamma_L/\beta_L$. On the other hand, for any $n\in\mathbb{N}$ and $M>0$,
        \[
            \mathbb{E}P^S_0\bigg(\frac{1}{n}\sum_{k=1}^n\abs{\zeta_{k}}\geq M\bigg) \leq e^{-Mn}\mathbb{E}E^S_0\big[e^{\sum_{k=1}^n|\zeta_{k}|}\big] \leq K^n e^{-Mn},
        \]
        where $K\leq C\mathbb{E}[e^{\zeta_{1}}]$ for some constant $C>0$, due to the correlation between two nearest-neighbor scatterers. Now we choose $M_0$ such that $M_0>\log K$. Then, for each $M\geq M_0$ and $N\in\mathbb{N}$,
        \[
            \sum_{n=N}^\infty \mathbb{E}P^S_0\bigg( \frac{1}{n}\sum_{k=1}^n \abs{\zeta_{k}} \geq M \bigg) \leq \sum_{n=N}^\infty K^n e^{-Mn} \leq \frac{K^Ne^{-MN}}{1-Ke^{-M}}<1.
        \]
        Define the event $A_{N,M}\coloneqq\{n^{-1}\sum_{k=1}^n\abs{\eta_{k}}\leq 2M,\,\forall~n\geq N \}$. Then by the above construction we have the estimate $\mathbb{E}\overline{P}^S_0(A_{N,M})\geq1- (1-Ke^{-M})^{-1}K^N e^{-MN}$. Thus, fixing $N_0$, we have $\mathbb{E}\overline{P}^S_0(\cup_{M\geq M_0}A_{N_0,M})=1$. And by the Dominated Convergence Theorem,
        \begin{equation}\label{eqn: dogdogdog}
            \lim_{N\to\infty}\frac{1}{N}\sum_{k=1}^N \zeta_{k} = \lim_{N\to\infty} \frac{1}{N} \sum_{k=1}^N E_Q[\eta_{k}] = \ell,\qquad\text{conditioned on}~A_{N_0,M},\quad\forall~M\geq M_0.
        \end{equation}
        Henceforth, the above (\ref{eqn: dogdogdog}) actually holds $\mathbb{P}$-a.s., which verifies the assertion.        
    \end{proof}

    \begin{proof}[Proof of Theorem \ref{thm: regular CLT}. Part I] 
        Notice that $X_n=\omega_{S_n}$ for each $n\in\mathbb{N}$. Now for any $x\in\mathbb{R}$ and any small $\epsilon>0$, we have $\mathbb{P}$-a.s. that $P^S_0( X_n- \ell E[S_1] n \leq x n^{1/2} )$ is less than or equal to
        \[
            P^S_0\big( \omega_{S_n} - \ell E[S_1]n \leq xn^{1/2},\, S_n - E[S_1]n > (\ell^{-1}x+\epsilon)n^{1/2} \big) + P^S_0\big( S_n - E[S_1]n \leq (\ell^{-1}x+\epsilon)n^{1/2} \big).
        \]
        Taking $n\to\infty$ and noticing that by Lemma \ref{lem: omega LLN}, $\omega_{S_n}\sim\omega_{S_n-E[S_1]n}+\ell E[S_1]n$ at large $n$, $\mathbb{P}$-a.s., we get
        \begin{equation*}\begin{aligned}
            &\varlimsup_{n\to\infty} P^S_0\big( X_n- \ell E[S_1] n \leq x n^{1/2} \big) \leq \varlimsup_{n\to\infty} P^S_0\big( \omega_{S_n-E[S_1]n}  \leq xn^{1/2},\, S_n - E[S_1]n > (\ell^{-1}x+\epsilon)n^{1/2} \big)\\
            &\qquad\quad+ \varlimsup_{n\to\infty} P^S_0\big( S_n - E[S_1]n \leq (\ell^{-1}x+\epsilon)n^{1/2} \big)\\
            &\qquad\leq \varlimsup_{n\to\infty} P^S_0 \big( \omega_{(\ell^{-1}x+\epsilon)n^{1/2}}  \leq xn^{1/2} \big) + \varlimsup_{n\to\infty} P^S_0\big( S_n - E[S_1]n \leq (\ell^{-1}x+\epsilon)n^{1/2} \big).
        \end{aligned}\end{equation*}
        Invoking the Lemma \ref{lem: omega LLN} and the CLT \cite{LeGall} for the long-range random walk $(S_n)_{n\geq0}$, the first term on the right-hand side vanishes and we then have
        \begin{equation}\label{eqn: 3.1}
            \varlimsup_{n\to\infty} P^S_0\big( X_n- \ell E[S_1] n \leq x n^{1/2} \big) \leq \frac{1}{\sqrt{2\pi}} \int_{-\infty}^{x^\prime} e^{-\frac{1}{2}t^2}\,dt,\qquad\text{where}\quad x^\prime\coloneqq E[S_1^2]^{-1/2}(\ell^{-1}x+\epsilon).
        \end{equation}
        In fact, an analogous derivation to (\ref{eqn: 3.1}) yields the lower-bound
        \[
            \varliminf_{n\to\infty} P^S_0\big( X_n- \ell E[S_1] n \leq x n^{1/2} \big) \geq \frac{1}{\sqrt{2\pi}} \int_{-\infty}^{x^{\prime\prime}} e^{-\frac{1}{2}t^2}\,dt,\qquad\text{with}\quad x^{\prime\prime}\coloneqq E[S_1^2]^{-1/2}(\ell^{-1}x-\epsilon).
        \]
        Combining the above two estimates, the CLT for discrete-time stochastic Lévy--Lorentz gas is verified.
    \end{proof}

\subsection{Continuous-time scenario.} We will need the following technical lemmas.
    \begin{lemma}\label{lem: T and n ratio limit}
        \normalfont
        Take $N_t\coloneqq\sup\{n\in\mathbb{N}:\,T^\omega_n\leq t\}\in\mathbb{N}\cup\{\infty\}$ for each $t\geq0$. We have $N_t<\infty$ for all $t>0$, $\mathbb{P}\otimes P^S_0$-a.s., and that
        \[
            \lim_{n\to\infty}\frac{T^\omega_n}{n}=\lim_{t\to\infty}\frac{t}{N_t}= E[\abs{V_1}]\ell,\qquad \mathbb{P}\otimes P^S_0\text{-a.s.,}
        \]
        which naturally implies $T^\omega_n\to\infty$, $\mathbb{P}\otimes P^S_0$-a.s., yielding that $(\widetilde{X}_t)_{t\geq0}$ is well-defined.
    \end{lemma}
    \begin{proof}
        Notice that the distance variables $(\zeta_r)_{r\in\mathbb{Z}}$ is not necessarily ergodic, so we cannot directly apply the Birkhoff theorem to produce the limit. Nevertheless, invoking \cite[Theorem 2.5.6]{Lalley}, we get
        \[
            \lim_{n\to\infty}\frac{1}{n}\sum_{k=1}^n\abs{X_k-X_{k-1}}=\mathbb{E} E^S_0[\abs{\omega_{S_1}}|\mathscr{L}] = \sum_{k\in\mathbb{Z}} P^S_0(V_1=k)\mathbb{E}[\abs{\omega_k}|\mathscr{L}],\qquad \mathbb{P}\otimes P^S_0\text{-a.s.,}
        \]
        where $\mathscr{L}$ denotes the invariant $\sigma$-algebra on $\mathbb{R}^\infty$. On the other hand, with the same \cite[Theorem 2.5.6]{Lalley} to $(n^{-1}\omega_n)_{n\geq1}$ yields $\mathbb{E}[\abs{\omega_k}|\mathscr{L}]=\lim_{n\to\infty} n^{-1}\omega_{\abs{k}n}=\ell\abs{k}$ for each $k\in\mathbb{Z}$. Thus, we have
        \[
            \lim_{n\to\infty}\frac{T^\omega_n}{n}=\lim_{n\to\infty}\frac{1}{n}\sum_{k=1}^n\abs{X_k-X_{k-1}} = \sum_{k\in\mathbb{Z}}\ell\abs{k} P^S_0(V_1=k) = E[\abs{V_1}]\ell,\qquad \mathbb{P}\otimes P^S_0\text{-a.s.,}
        \]
        which also indicates $T^\omega_n\to\infty$ as $n\to\infty$ and then that $N_t<\infty$ for all $t>0$. Since $N_t\to\infty$ as $t\to\infty$,
        \[
            \lim_{t\to\infty}\frac{T^\omega_{N_t+1}}{N_t} = \lim_{t\to\infty}\frac{T^\omega_{N_t+1}}{N_t+1}\vdot\frac{N_t+1}{N_t} = \lim_{t\to\infty}\frac{T^\omega_{N_t}}{N_t} = E[\abs{V_1}]\ell,\qquad \mathbb{P}\otimes P^S_0\text{-a.s.}
        \]
        Taking the above into the following consideration
        \[
            N_t^{-1}\abs{T^\omega_{N_t}-t}\leq N_t^{-1}(T^{\omega}_{N_t+1}-T^{\omega}_{N_t})\to0\quad\text{as}\;\;t\to\infty,
        \]
        the assertion is verified.
    \end{proof}

    \begin{lemma}\label{lem: Cesaro sum limit}
        \normalfont
        We have the following Cesàro-type limit,
        \[
            \lim_{n\to\infty} \sum_{k\in\mathbb{Z}} \zeta_{k+\beta} P^S_0(S_n=k) = \ell,\qquad\forall~\beta\in\mathbb{Z},\qquad\mathbb{P}\text{-a.s.,}
        \]
        uniformly on $\abs{\beta}\leq\psi_n$ with any $\psi_n\sim o(n^{1/2})$, for the mutually interacting distances $(\zeta_r)_{r\in\mathbb{Z}}$.
    \end{lemma}
    \begin{proof}
        For each $k_1\leq k_2$ in $\mathbb{Z}$, define $\psi_{k_1,k_2}\coloneqq(k_2-k_1+1)^{-1}\sum_{k=k_1}^{k_2}\zeta_{k+\beta}$. For all $n\in\mathbb{N}$ and $r\in\mathbb{Z}$, define
        \[
            \mathscr{E}^+_{n,r}\coloneqq\sum_{k>r} \zeta_{k+\beta} P^S_0(S_n=k),\qquad \mathscr{E}^{-}_{n,r}\coloneqq\sum_{k<-r}\zeta_{k+\beta} P^S_0(S_n=k)
        \]
        with $\mathscr{E}^+_n\coloneqq\mathscr{E}^+_{n,-1}$, $\mathscr{E}^-_{n}\coloneqq\mathscr{E}^-_{n,0}$ and $\mathscr{E}_n\coloneqq\mathscr{E}^+_n+\mathscr{E}^-_{n}$. A straightforward rearrangement yields that
        \begin{equation}\label{eqn: 3.2}
            \mathscr{E}^+_{n,r} = \sum_{k>r} \psi_{0,k}(k+1)(P^S_0(S_n=k) - P^S_0(S_n=k+1)),\;\; \mathscr{E}^-_{n,r} = \sum_{k<-r} \psi_{k,-1}(-k)(P^S_0(S_n=k) - P^S_0(S_n=k-1)).
        \end{equation}
        Invoking Lemma \ref{lem: omega LLN},  for any $\epsilon>0$ we can find some $r=r(\omega)\in\mathbb{N}$ such that
        \[
            \abs{\psi_{0,k}-\ell}<\epsilon,\qquad\abs{\psi_{-k,-1}-\ell}<\epsilon,\qquad\forall~k\geq r.
        \]
        Assigning this particular value $r(\omega)$ to the following $r\in\mathbb{N}$ from now on, we get
        \[
            \abs{\mathscr{E}^+_{n,r} - \ell P^S_0(S_n>r)}\leq \epsilon P^S_0(S_n>r)\leq\epsilon,\qquad \abs{\mathscr{E}^-_{n,r} - \ell P^S_0(S_n<-r)}\leq \epsilon P^S_0(S_n<-r)\leq\epsilon,
        \]
        which yields $\abs{\mathscr{E}_{n,r} - \ell P^S_0(\abs{S_n}>r)}\leq \epsilon P^S_0(\abs{S_n}>r)\leq\epsilon$, where $\mathscr{E}_{n,r}\coloneqq\mathscr{E}^+_{n,r}+\mathscr{E}^-_{n,r}$. Invoking the CLT of the underlying random walk $(S_n)_{n\geq0}$, we can find some $N=N(\omega)\in\mathbb{N}$ such that
        \[
            P^S_0(\abs{S_n}\leq r)\leq\epsilon,\qquad\forall~n\geq N.
        \]
        We also denote $\normx{\zeta}_r\coloneqq\max_{\abs{k}\leq r}\abs{\zeta_{k+\beta}}$. Henceforth, with (\ref{eqn: 3.2}), for any $n\geq N$ we have
        \begin{equation*}\begin{aligned}
            \abs{\mathscr{E}_n-\ell} &\leq\abs{\mathscr{E}_n-\mathscr{E}_{n,r}} + \abs{\mathscr{E}_{n,r}-\ell P^S_0(\abs{S_n}>r)} + \abs{\ell P^S_0(\abs{S_n}>r)-\ell}\\
            &\leq (\normx{\zeta}_r+\ell)P^S_0(\abs{S_n}\leq r) + \abs{\mathscr{E}_{n,r}-\ell P^S_0(\abs{S_n}>r)}\leq \epsilon(\normx{\zeta}_r+\ell+1)\to0,\qquad\mathbb{P}\text{-a.s.,}
        \end{aligned}\end{equation*}
        which verifies the assertion.
    \end{proof}
    \begin{proof}[Proof of Theorem \ref{thm: regular CLT}. Part II]
        Take the deterministic $\Bar{N}_t\coloneqq[t/E[\abs{V_1}]\ell]$ for all $t\geq0$. Invoking the CLT for $(X_n)_{n\geq0}$, which has been verified in \textit{Part I.} of the theorem, we get
        \begin{equation}\label{eqn: 3.3}
            \lim_{t\to\infty} \frac{1}{\sqrt{t}}(X_{\Bar{N}_t}-\ell E[S_1]\Bar{N}_t) \stackrel{\mathcal{L}}{=} \mathcal{N}(0,\ell E[S_1^2]/E[\abs{V_1}]),\qquad\mathbb{P}\text{-a.s.}
        \end{equation}
        Since $\widetilde{X}_t-\ell E[S_1]t$ always locates between $X_{N_t}-\ell E[S_1]N_t$ and $X_{N_t+1}-\ell E[S_1](N_t+1)$ for all $t$, we have
        \begin{equation*}\begin{aligned}
            &\big|(\widetilde{X}_t-\ell E[S_1]t) - (X_{\Bar{N}_t}-\ell E[S_1]\Bar{N}_t)\big| \leq \big|(X_{N_t}-\ell E[S_1]N_t) - (X_{\Bar{N}_t}-\ell E[S_1]\Bar{N}_t)\big|\\ 
            &\qquad+ \big|(X_{N_t+1}-\ell E[S_1](N_t+1)) - (X_{\Bar{N}_t}-\ell E[S_1]\Bar{N}_t)\big|.
        \end{aligned}\end{equation*}
        In light of (\ref{eqn: 3.3}) and the Slutzky's theorem \cite[Theorem 13.18]{Klenke}, it suffices to show that the above terms converge in law to zero, $\mathbb{P}$-a.s. In turn we apply the Portemanteau's theorem \cite[Theorem 13.16]{Klenke}, it is then enough to show
        \begin{equation}\label{eqn: 3.4}        
            E^S_0\big[f(t^{-1/2}\abs{(X_{N_t}-\ell E[S_1]N_t) - (X_{\Bar{N}_t}-\ell E[S_1]\Bar{N}_t)})\big]\to f(0)\qquad\text{as}\quad t\to\infty,\qquad\mathbb{P}\text{-a.s.}
        \end{equation}
        for any bounded Lipschitz $f:\mathbb{R}\to\mathbb{R}$. Notice that we have omitted the analysis regarding the term $X_{N_t+1}-\ell E[S_1](N_t+1)$, which is completely analogous to the above (\ref{eqn: 3.4}). Take any function $\psi:\mathbb{R}_+\to\mathbb{R}_+$ such that $\psi_t\to\infty$ and $t^{-1}\psi_t\to0$ as $t\to\infty$. Invoking Lemma \ref{lem: T and n ratio limit}, we know $t^{-1}(N_t-\Bar{N}_t)\to0$ as $t\to\infty$. So for any $\epsilon>0$, there exists some $t_0=t_0(\omega)>0$ and event $\mathcal{A}=\mathcal{A}(\omega)$ with $P^S_0(\mathcal{A})\leq\epsilon/4$ such that
        \[
            \abs{N_t-\Bar{N_t}}\leq \psi_t,\qquad\forall~t\geq t_0\qquad\text{on}\quad\Omega\backslash\mathcal{A},\qquad\mathbb{P}\text{-a.s.}
        \]
        On the other hand, by the CLT of the random walk $(S_n)_{n\geq0}$, we can find some $t_1=t_1(\omega)>0$ and event $\mathcal{B}=\mathcal{B}(\omega)$ with $P^S_0(\mathcal{B})\leq\epsilon/4$ such that
        \[
            \abs{S_{N_t}-S_{\Bar{N}_t}}\leq \abs{N_t-\Bar{N_t}}^{1/2},\qquad\forall~t\geq t_1\qquad\text{on}\quad\Omega\backslash\mathcal{B},\qquad\mathbb{P}\text{-a.s.}
        \]
        Henceforth, to verify (\ref{eqn: 3.4}), for any $t\geq t_0\vee t_1$ we first write
        \[
            E^S_0\big[f(t^{-1/2}\abs{(X_{N_t}-\ell E[S_1]N_t) - (X_{\Bar{N}_t}-\ell E[S_1]\Bar{N}_t)}),\,\mathcal{A}\cup\mathcal{B}\big] \leq 2\normx{f}_\infty P^S_0(\mathcal{A}\cup\mathcal{B})\leq \normx{f}_\infty\epsilon,
        \]
        $\mathbb{P}$-a.s. Furthermore, for any $t\geq0$ we have
        \begin{equation*}\begin{aligned}
            &E^S_0\big[f(t^{-1/2}\abs{(X_{N_t}-\ell E[S_1]N_t) - (X_{\Bar{N}_t}-\ell E[S_1]\Bar{N}_t)}),\,\Omega\backslash(\mathcal{A}\cup\mathcal{B})\big]\\
            &\qquad\leq \text{Lip}(f) t^{-1/2} E^S_0[\sum_{\beta\in K_t}\zeta_{S_{\Bar{N}_t}+\beta}]\leq\text{Lip}(f)t^{-1/2}(\psi_t^{1/2}+1)\normx{E^S_0[\zeta_{\vdot+S_{\Bar{N}_t}}]}_{K_t},\qquad\mathbb{P}\text{-a.s.,}
        \end{aligned}\end{equation*}
        where $K_t\coloneqq\mathbb{N}\cap[-|S_{N_t}-S_{\Bar{N}_t}|,|S_{N_t}-S_{\Bar{N}_t}|]$ and $\normx{E^S_0[\zeta_{\vdot+S_{\Bar{N}_t}}]}_{K_t}$ abbreviates $\sup_{\beta\in K_t} E^S_0[\zeta_{S_{\Bar{N}_t}+\beta}]$. In light of Lemma \ref{lem: Cesaro sum limit}, we can further say
        \begin{equation*}\begin{aligned}
            &E^S_0\big[f(t^{-1/2}\abs{(X_{N_t}-\ell E[S_1]N_t) - (X_{\Bar{N}_t}-\ell E[S_1]\Bar{N}_t)}),\,\Omega\backslash(\mathcal{A}\cup\mathcal{B})\big]\leq 2\text{Lip}(f)\ell(t^{-1}\psi_t)^{1/2}\leq \text{Lip}(f)\ell\epsilon,
        \end{aligned}\end{equation*}
        for all $t\geq t_2$, $\mathbb{P}$-a.s. for some $t_2=t_2(\omega)>0$. Combining the above estimates on the disjoint events $\mathcal{A}\cup\mathcal{B}$ and $\Omega\backslash(\mathcal{A}\cup\mathcal{B})$, we get
        \[
            E^S_0\big[f(t^{-1/2}\abs{(X_{N_t}-\ell E[S_1]N_t) - (X_{\Bar{N}_t}-\ell E[S_1]\Bar{N}_t)})\big] \leq (\normx{f}_\infty+\text{Lip}(f)\ell)\epsilon,\qquad\forall~t\geq\max\{t_0,t_1,t_2\},
        \]
        $\mathbb{P}$-a.s., and the assertion (\ref{eqn: 3.4}) is then verified.
    \end{proof}

\subsection{Step-reinforced scenario.}\label{sec: 3.3} 
    We will need the following technical lemmas.
    \begin{lemma}\label{lem: martingale approach}
        \normalfont
        When $0\leq p<3/4$, we have the following asymptotic normality,
        \[
            \lim_{n\to\infty} P^{(p)}_0\big(\tfrac{1}{\sqrt{n}}S^{(p)}_n\leq (3-4p)^{-1/2}x\big) = \frac{1}{\sqrt{2\pi}}\int_{-\infty}^x e^{-\frac{1}{2}t^2}\,dt,\qquad\forall~x\in\mathbb{R}.
        \]
    \end{lemma}
    \begin{proof}
        Write $\mathcal{F}_n\coloneqq\sigma(V^{(p)}_1,\ldots,V^{(p)}_n)$ and $\mathcal{U}_n\sim\mathcal{U}\{1,\ldots,n\}$ for each $n\in\mathbb{N}$. Clearly we have
        \[
            E^{(p)}_0[V^{(p)}_{n+1}|\mathcal{F}_n] = (2p-1)E^{(p)}_0[V^{(p)}_{\mathcal{U}_n}|\mathcal{F}_n] = \frac{1}{n}(2p-1)\sum_{k=1}^n V^{(p)}_k = \frac{a}{n}S^{(p)}_n,
        \]
        where $a\coloneqq(2p-1)$. Henceforth, for each $n\in\mathbb{N}$, 
        \begin{equation}\label{eqn: 3.5}        
            E^{(p)}_0[S^{(p)}_{n+1}|\mathcal{F}_n] = \gamma_n S^{(p)}_n,\qquad\text{where}\quad\gamma_n\coloneqq1+\frac{a}{n}.
        \end{equation}
        Now we define the sequence $(M^{(p)}_n)_{n\geq0}$ via $M_n\coloneqq a_n S^{(p)}_n$ for all $n\geq0$, where $a_0=a_1\coloneqq1$ and
        \[
            a_n\coloneqq \prod_{k=1}^{n-1} \gamma_k^{-1} = \frac{\Gamma(a+1)\Gamma(n)}{\Gamma(n+a)},\qquad\forall~n\geq2.
        \]
        Following from the well-known property \cite{Davis} of Euler--Gamma functions, one observes $\Gamma(n+a)\sim n^{2p-1}\Gamma(n)$ as $n\to\infty$. And thus we have $n^{2p-1}a_n\to\Gamma(a+1)$ as $n\to\infty$. Furthermore, since $a_n=\gamma_n a_{n+1}$, we can deduce from (\ref{eqn: 3.5}) that
        \[
            E^{(p)}_0[M^{(p)}_{n+1}|\mathcal{F}_n] = M^{(p)}_n,\qquad P^{(p)}_0\text{-a.s.,}\qquad\forall~n\geq1.
        \]
        Thus $(M^{(p)}_n)_{n\geq0}$ is a $(\mathcal{F}_n)_{n\geq0}$-adapted martingale, via similar approach in \cite{Bercu}. Notice the martingale can be rewritten as 
        \[
            M^{(p)}_n = \sum_{k=1}^n a_k\epsilon_k,\qquad\text{with}\quad \epsilon_n\coloneqq S^{(p)}_n-\gamma_{n-1}S^{(p)}_{n-1},\qquad\forall~n\geq1,
        \]
        since its increments $\Delta M^{(p)}_n=a_nS^{(p)}_n-a_{n-1}S^{(p)}_{n-1}=a_n\epsilon_n$. The predictable quadratic variation associated with $(M^{(p)}_n)_{n\geq0}$ is given via
        \[
            \langle M\rangle_n^{(p)}\coloneqq\sum_{k=1}^n E^{(p)}_0\big[(\Delta M^{(p)}_k)^2\big|\mathcal{F}_{k-1}\big],\qquad\forall~n\geq1.
        \]
        From the martingale property, we already know $E^{(p)}_0[\epsilon_{n+1}|\mathcal{F}_n]=0$. Moreover, following (\ref{eqn: 3.5}) we have
        \begin{equation}\label{eqn: 3.6}
            E^{(p)}_0\big[(S^{(p)}_{n+1})^2\big|\mathcal{F}_n\big] = E^{(p)}_0\big[(S^{(p)}_{n})^2\big|\mathcal{F}_n\big] + \tfrac{2a}{n}(S^{(p)}_n)^2+1 = (1+\tfrac{2a}{n})(S^{(p)}_n)^2+1,
        \end{equation}
        which ensures that
        \[
            E^{(p)}_0\big[(\epsilon_{n+1})^2\big|\mathcal{F}_n\big] = E^{(p)}_0\big[(S^{(p)}_{n+1})^2\big|\mathcal{F}_n\big] - \gamma_n^2(S^{(p)}_n)^2 = 1 - (\gamma_n-1)^2(S^{(p)}_n)^2,\qquad P^{(p)}_0\text{-a.s.}
        \]
        By the same token,
        \begin{equation*}\begin{aligned}
            E^{(p)}_0\big[(\epsilon_{n+1})^4\big|\mathcal{F}_n\big] &= 1 - 3(\gamma_n-1)^4(S^{(p)}_n)^4 + 2(\gamma_n-1)^2(S^{(p)}_n)^2\\
            &\leq 4/3 - 3\big( (\gamma_n-1)^2(S^{(p)}_n)^2-1/3 \big)^2,\qquad P^{(p)}_0\text{-a.s.}
        \end{aligned}\end{equation*}
        Consequently, we obtain the $P^{(p)}_0$-a.s. upper-bounds
        \begin{equation}\label{eqn: 3.7}
            \sup_{n\geq0} E^{(p)}_0[(\epsilon_{n+1})^2|\mathcal{F}_n]\leq1\qquad\text{and}\qquad \sup_{n\geq0}E^{(p)}_0[(\epsilon_{n+1})^4|\mathcal{F}_n]\leq 4/3.
        \end{equation}
        Henceforth, we deduce from (\ref{eqn: 3.6}) that
        \begin{equation}\label{eqn: 3.8}
            \langle M\rangle^{(p)}_n = a_1^2 E^{(p)}_0[(\epsilon_1)^2] + \sum_{k=1}^{n-1} a_{k+1}^2 E^{(p)}_0[(\epsilon_{k+1})^2|\mathcal{F}_k] = \sum_{k=1}^n a_k^2 - a^2\sum_{k=1}^{n-1}(a_{k+1}/k)^2(S^{(p)}_k)^2
        \end{equation}
        whose asymptotic behavior is closely related to the one of $\nu_n\coloneqq\sum_{k=1}^n a_k^2$ for all $n\geq1$. One can observe that we always have $\langle M\rangle^{(p)}_n\leq\nu_n$. When $p<3/4$, in particular, we have
        \[
            n^{2a-1}\nu_n\to(1-2a)^{-1}\Gamma(a+1)^2 \qquad\text{as}\quad n\to\infty,\qquad P^{(p)}_0\text{-a.s.}
        \]
        Besides the above asymptotics of $(\nu_n)_{n\geq1}$, in light of \cite[Theorem 4.3.15]{Duflo}, one gets
        \[
            |M^{(p)}_n|^2/\langle M\rangle^{(p)}_n = o\big((\log \langle M\rangle^{(p)}_n)^{1+\epsilon}\big),\qquad P^{(p)}_0\text{-a.s.,}\qquad\forall~\epsilon>0.
        \]
        Consequently,
        \[
            |M^{(p)}_n|^2 = o\big(\nu_n (\log \nu_n )^{1+\epsilon}\big)\quad\Longrightarrow\quad |S^{(p)}_n|^2 = o\big(n (\log n )^{1+\epsilon}\big),\qquad P^{(p)}_0\text{-a.s.,}\qquad\forall~\epsilon>0,
        \]
        since $S^{(p)}_n=a_n^{-1}M^{(p)}_n$ for each $n\geq1$. In view of (\ref{eqn: 3.8}), we have
        \begin{equation}\label{eqn: 3.9}
            \lim_{n\to\infty}\nu_n^{-1}\langle M\rangle^{(p)}_n=1,\qquad \sum_{n=1}^\infty \nu_n^{-2}a_n^4<\infty,\qquad P^{(p)}_0\text{-a.s.,}
        \end{equation}
        where the last result follows from the asymptotics of $(\nu_n)_{n\geq1}$. In light of (\ref{eqn: 3.9}), to show the asymptotic normality it suffices to verify that $(M^{(p)}_n)_{n\geq0}$ admits the Lindeberg's condition \cite[Corollary 3.1]{Hall/Heyde}. Namely, we will show
        \begin{equation}
            \nu^{-1}_n\sum_{k=1}^n E^{(p)}_0\big[ (\Delta M^{(p)}_k)^2\mathbbm{1}_{\{ |\Delta M^{(p)}_k|\geq\epsilon\nu_n^{1/2} \}} \big|\mathcal{F}_{k-1}\big]\to0\qquad\text{in}\quad P^{(p)}_0\text{-probability},\qquad\forall~\epsilon>0.
        \end{equation}
        Indeed, for any $\epsilon>0$, we have from (\ref{eqn: 3.7}) that 
        \begin{equation*}\begin{aligned}
            &\nu^{-1}_n\sum_{k=1}^n E^{(p)}_0\big[ (\Delta M^{(p)}_k)^2\mathbbm{1}_{\{ |\Delta M^{(p)}_k|\geq\epsilon\nu_n^{1/2} \}} \big|\mathcal{F}_{k-1}\big] \leq (\epsilon\nu_n)^{-2}\sum_{k=1}^n E^{(p)}_0\big[ (\Delta M^{(p)}_k)^4 \big|\mathcal{F}_{k-1}\big]\\
            &\qquad\leq \sup_{1\leq k\leq n} E^{(p)}_0[(\epsilon_k)^4|\mathcal{F}_{k-1}]\vdot(\epsilon\nu_n)^{-2}\sum_{k=1}^n a_k^4\leq 4(\sqrt{3}\epsilon\nu_n)^{-2}\sum_{k=1}^na_k^4.
        \end{aligned}\end{equation*}
        However, we already see from (\ref{eqn: 3.9}) that $\sum_{n=1}^\infty\nu_n^{-2}a_n^4<\infty$. Hence via the Kronecker's lemma \cite[Lemma 4.3.2]{Shiryaev}, we see that $\nu^{-2}_n\sum_{k=1}^na_k^4\to0$ as $n\to\infty$, which ensures the Lindeberg's condition is satisfied. Hence, we conclude from the martingale CLT that
        \[
            \tfrac{1}{\sqrt{\nu_n}}M^{(p)}_n\xrightarrow{\mathcal{L}} \mathcal{N}(0,1)\qquad\text{as}\quad n\to\infty.
        \]
        Because $M^{(p)}_n=a_nS^{(p)}_n$ and $na_n^2\sim (1-2a)\nu_n$ at large $n$, the assertion immediately follows.
    \end{proof}
    \begin{proof}[Proof of Theorem \ref{thm: regular reinforced CLT}.]
        The regular CLT for the step-reinforced Lévy--Lorentz gas $(X^{(p)}_n)_{n\geq0}$ now follows from Lemma \ref{lem: omega LLN} via an analogous argument to \textit{Part I.} of the proof of Theorem \ref{thm: regular CLT}, where we use Lemma \ref{lem: martingale approach} instead of the CLT for the Markovian (memoryless) random walk $(S_n)_{n\geq0}$. Repeating the steps in Lemma \ref{lem: T and n ratio limit}, one observes that
        \begin{equation}\label{eqn: 3.11}
            \varliminf_{n\to\infty} \frac{1}{n}\sum_{k=1}^n |X^{(p)}_k-X^{(p)}_{k-1}|>0,\qquad\mathbb{P}\otimes P^{(p)}_0\text{-a.s.}
        \end{equation}
        And thus the continuous-time step-reinforced Lévy--Lorentz gas $(\widetilde{X}^{(p)}_t)_{t\geq0}$ is indeed well-defined on $[0,\infty)$. A straightforward computation yields the step-reinforced random walk $(S^{(p)}_n)_{n\geq0}$ is measure-preserving on cylinder sets, see also \cite{Collevecchio/Hamza/Shi/Williams}. Hence invoking \cite[Theorem 2.5.6]{Lalley}, we still have
        \[
            \lim_{n\to\infty}\frac{1}{n}\sum_{k=1}^n |X^{(p)}_k-X^{(p)}_{k-1}| = \lim_{n\to\infty}\frac{T^{(p)}_n}{n} = \lim_{t\to\infty}\frac{t}{N^{(p)}_t} = E^{(p)}_0[|V^{(p)}_1|]\ell,\qquad\mathbb{P}\otimes P^{(p)}_0\text{-a.s.,}
        \]
        where $N^{(p)}_t\coloneqq\sup\{n\geq1:\, T^{(p)}_n\leq t\}$. Using again Lemma \ref{lem: martingale approach} instead of the CLT of $(S_n)_{n\geq0}$, we further have, with $\beta\in\mathbb{Z}$, that 
        \[
            \sum_{k\in\mathbb{Z}} \zeta_{k+\beta}P^{(p)}_0(S^{(p)}_n=k)\to\ell\qquad\text{as}\quad n\to\infty,\qquad\mathbb{P}\text{-a.s.,}
        \]
        uniformly on $\abs{\beta}\leq\psi_n$ with any $\psi_n\sim o(n^{1/2})$. And henceforth we have achieved the asymptotic normality for $(\widetilde{X}^{(p)}_t)_{t\geq0}$ at $p<3/4$.
    \end{proof}

\section{Functional central limit theorem and Skorokhod space}
    In this section, we derive the scaling limits of the stochastic Lévy--Lorentz gases $(X_n)_{n\geq0}$ and $(\widetilde{X}_t)_{t\geq0}$, as well as their step-reinforced version $(X^{(p)}_n)_{n\geq0}$ and $(\widetilde{X}^{(p)}_n)_{n\geq0}$, viewed from the Skorokhod space of càdlàg paths.

\subsection{Extension of Skorokhod space topology.} 
    The classic setting of the Skorokhod topology only deals with real functions on the half line $[0,\infty)$. In light that the traveling Markovian particle $(S_n)_{n\geq0}$ and the step-reinforced particle $(S^{(p)}_0)_{n\geq0}$ are allowed to take negative values, we extend the Skorokhod space in suitable manner. {The extension in this section is already introduced in the literature, for instance \cite{Stivanello/Bet/Bianchi/Lenci/Magnanini}. For pedological consistency, we still give a presentation here.}\par
    Let $\mathcal{D}(\mathbb{R}\to\mathbb{R}^d)$ denote the space of $\mathbb{R}^d$-valued càdlàg functions on $\mathbb{R}$, and we use $\mathcal{D}(\mathbb{R}_+\to\mathbb{R}^d)$ for the restriction of all such càdlàg functions on the half line $\mathbb{R}_+=[0,\infty)$. Now we endow $\mathcal{D}(\mathbb{R}\to\mathbb{R}^d)$ with the $J_1$-Skorokhod topology adapted from the usual space $\mathcal{D}(\mathbb{R}_+\to\mathbb{R}^d)$. To this end, let $\Lambda$ be the space of all strictly increasing continuous functions $\lambda:\mathbb{R}\to\mathbb{R}$ satisfying $\lambda_t\searrow-\infty$ as $t\searrow-\infty$ and $\lambda_t\nearrow\infty$ as $t\nearrow\infty$. And we have the following result.\par

    \begin{lemma}
        \normalfont
        We can find a metrizable topology on $\mathcal{D}(\mathbb{R}\to\mathbb{R}^d)$, called the $J_1$-Skorokhod topology, under which $\mathcal{D}(\mathbb{R}\to\mathbb{R}^d)$ is a Polish space and which is characterized in the following sense: A sequence $\{x_n\}\subseteq\mathcal{D}(\mathbb{R}\to\mathbb{R}^d)$ of càdlàg functions converges to $x\in\mathcal{D}(\mathbb{R}\to\mathbb{R}^d)$ in this topology if and only if there exists $\{\lambda_n\}\subseteq\Lambda$ such that $\normx{I-\lambda_n}_{\mathbb{R}}\to0$ as $n\to\infty$ and $\normx{x-x_n\circ\lambda_n}_M\to0$ as $n\to\infty$ for any $M>0$. Furthermore, if $x$ is continuous, then $x_n\to x$ in $J_1$-Skorokhod topology if and only if $x_n\to x$ in the almost uniform topology.
    \end{lemma}
    \begin{proof}
        See \cite[Theorem 1.14]{Jacod/Shiryaev} for a proof. Notice that we have used $I$ for the identity map on $\mathbb{R}^d$, and $\normx{\vdot}_M$ for the sup-norm of all functions restricted to $[-M,M]$. See also \cite{Brace} for a discussion on the almost uniform topology.  
    \end{proof}
    See \cite{Chen/Margarint1,Chen/Margarint2} for more discussions on continuous processes. A more tangible description of the $J_1$-Skorokhod topology on $\mathcal{D}(\mathbb{R}\to\mathbb{R}^d)$ compels us for an explicit metric that is compatible with this topology. We first choose the function $k_N(\vdot)$ on $\mathbb{R}$ by specifying $k_N(t)=1$ when $\abs{t}\leq N$, $k_N(t)=N-t+1$ when $N<t\leq N+1$, $k_N(t)=N+t+1$ when $-(N+1)\leq t<-N$, and $k_N(t)$ vanishes whenever $\abs{t}>N+1$. We also define the functional $\normx{\lambda}_{\mathscr{D}}\coloneqq\sup\{ |\log (t-s)^{-1}(\lambda_t-\lambda_s)|:\,t>s \}$ on $\Lambda$. Now for all $x,y$ in $\mathcal{D}(\mathbb{R}\to\mathbb{R}^d)$, we define
    \begin{equation}\label{eqn: 4.1}    
        d(x,y)\coloneqq\sum_{N=1}^\infty 2^{-N} (\delta_N(x,y)\wedge1),\qquad\text{where}\quad \delta_N(x,y)\coloneqq \inf_{\lambda\in\Lambda}\big\{ \normx{(k_Nx)\circ\lambda-k_Ny}_{\mathbb{R}} + \normx{\lambda}_{\mathscr{D}} \big\},
    \end{equation}
    where we naturally write $(k_Nx)(t)=k_N(t)x(t)$ for all $t\in\mathbb{R}$.
    \begin{lemma}
        \normalfont
        The functional $d(\vdot)$ defined in (\ref{eqn: 4.1}) is a metric on $\mathcal{D}(\mathbb{R}\to\mathbb{R}^d)$ metrizing the $J_1$-Skorokhod topology. And in view of this metric, the space $\mathcal{D}(\mathbb{R}\to\mathbb{R}^d)$ is complete and separable, i.e.~Polish space.
    \end{lemma}
    \begin{proof}
        See \cite[Chapter VI.]{Jacod/Shiryaev} for a proof.
    \end{proof}
    Now we use $\mathcal{D}_d$ to abbreviate $\mathcal{D}(\mathbb{R}\to\mathbb{R}^d)$ for all dimensions $d\geq1$. In particular, we shall use $\mathscr{B}$ for the Borel $\sigma$-algebra on $\mathcal{D}_1$ and $\mathscr{B}^{2}$ for the Borel $\sigma$-algebra on $\mathcal{D}_2$. Let $\mathcal{D}_0$ denote the space of all nondecreasing real functions on $\mathbb{R}$. It is obvious that $\mathcal{D}_0\subseteq\mathcal{D}_1$ and $\mathscr{B}_0=\mathcal{D}_0\cap\mathscr{B}$, where $\mathscr{B}_0$ denotes the Borel $\sigma$-algebra on $\mathcal{D}_0$. We will also need the following transform $\mathcal{K}$ defined by
    \[
        \mathcal{K}:\mathcal{D}_0\times\mathcal{D}_1\to\mathcal{D}_1\qquad\text{via}\quad \mathcal{K}:(x,y)\in\mathcal{D}_0\times\mathcal{D}_1\mapsto x\circ y\in\mathcal{D}_1.
    \]
    And we have the following propositions.
    \begin{lemma}\label{lem: skorokhod 3}
        \normalfont
        The transform $\mathcal{K}:\mathscr{B}_0\otimes\mathscr{B}\to\mathscr{B}$ is measurable with respect to the Borel $\sigma$-algebras. And the same transform $\mathcal{K}$ is continuous when restricted to the domain $(\mathcal{C}\cap\mathcal{D}_0)\times\mathcal{D}_1$, where $\mathcal{C}$ denotes the space of all continuous real functions on $\mathbb{R}$.
    \end{lemma}
    \begin{proof}
        We first show the Borel measurability of $\mathscr{K}$ on $\mathcal{D}_0\times\mathcal{D}_1$. Since the $\sigma$-algebras $\mathscr{B}_0$ and $\mathscr{B}$ are generated by finite-dimensional cylinder sets, it suffices to show that for any time point $t\in\mathbb{R}$ the mapping $(x,y)\in\mathcal{D}_0\times\mathcal{D}_1\mapsto x(y_t)\in\mathbb{R}$ is $\mathbb{R}$-Borel measurable with respect to $\mathscr{B}_0\times\mathscr{B}$. Namely, we only need to show the event $\{(x,y) \in \mathcal{D}_0\times\mathcal{D}_1 :\, x(y_t)\leq s \}\in\mathscr{B}_0\otimes\mathscr{B}$ for any height $s\in\mathbb{R}$ and time point $t\in\mathbb{R}$. Indeed, denote $y_{k,t}\coloneqq k^{-1}\min\{ j\in\mathbb{Z}:\,k y_t\leq j \}$ for each $k\in\mathbb{N}$ and $y\in\mathcal{D}_1$. Then, $y_{k,t}\to y_t$ pointwise in $t\in\mathbb{R}$ as $k\to\infty$. Hence the mappings $(x,y)\in\mathcal{D}_0\times\mathcal{D}_1\mapsto x(y_{k,t})$ converge pointwise in $t$ to the mapping $(x,y)\in\mathcal{D}_0\times\mathcal{D}_1\mapsto x(y_t)$ as $k\to\infty$. It then suffices to show the $\mathbb{R}$-Borel measurability of this mapping sequence with respect to $\mathscr{B}_0\otimes\mathscr{B}$. Namely, we only need to show the event $\{(x,y) \in \mathcal{D}_0\times\mathcal{D}_1 :\, x(y_{k,t})\leq s \}\in\mathscr{B}_0\otimes\mathscr{B}$ for each $k\in\mathbb{N}$ and height $s\in\mathbb{R}$. Denote the events $A_{k,j}\coloneqq\{(x,y):\,j-1\leq ky_t<j\}\cap\{(x,y):\,x(y_t)\leq s\}$ for each $j\in\mathbb{Z}$. Notice that for any Borel set $A\in\mathscr{B}$ we have $\{(x,y):\,y_t\in A\}\in\mathscr{B}_0\otimes\mathscr{B}$ and that for any height $s\in\mathbb{R}$ we have $\{(x,y):\,x_t\leq s\}\in \mathscr{B}_0 \otimes \mathscr{B}$. Henceforth, $A_{k,j}=\{(x,y):\,j-1\leq ky_t<j\}\cap\{(x,y):\,x_{j/k}\leq s\}\in\mathscr{B}_0\otimes\mathscr{B}$. In light of that $\{(x,y):\,x(y_{k,t})\leq s\}=\cup_{j\in\mathbb{Z}}A_{k,j}$, we have verified the assertion via the countable addictivity of $\sigma$-algebras. We now show the continuity of $\mathcal{K}$ when restricted to $(\mathcal{C}\cap\mathcal{D}_0)\times\mathcal{D}_1$. Choose a sequence $\{(x_n,y_n)\}\subseteq(\mathcal{C}\cap\mathcal{D}_0)\times\mathcal{D}_1$ such that $(x_n,y_n)\to(x,y)\in(\mathcal{C}\cap\mathcal{D}_0)\times\mathcal{D}_1$ in the $J_1$-Skorokhod topology. Then we can find a sequence $\{\lambda_n\}\subseteq\Lambda$ such that $\normx{Id-\lambda_n}_{\mathbb{R}}\to0$ as $n\to\infty$ and $(x_n\circ\lambda_n,y_n\circ\lambda_n)\to(x,y)$ as $n\to\infty$ uniformly on compacts in $\mathbb{R}^2$. For any $a,b\in\mathbb{R}$, $\normx{x_n\circ y_n\circ\lambda_n-x\circ y}_{[a,b]}\leq \normx{x_n\circ y_n\circ\lambda_n-x\circ y_n\circ\lambda_n}_{[a,b]}+ \normx{x\circ y_n\circ\lambda_n-x\circ y}_{[a,b]}\leq \normx{x_n-x}_{[a^\prime,b^\prime]} + \normx{x\circ y_n\circ\lambda_n-x\circ y}_{[a,b]}$ where $a^\prime\coloneqq\inf_{n\in\mathbb{N}}\inf_{a\leq t\leq b}\abs{y_n\circ\lambda_n}(t)$ and $b^\prime\coloneqq\sup_{n\in\mathbb{N}}\sup_{a\leq t\leq b}\abs{y_n\circ\lambda_n}(t)$. Since $x_n\to x$ and $y_n\circ\lambda_n\to y$ as $n\to\infty$ uniformly on compacts, the most left-hand side decays to zero, and hence the continuity assertion is verified.
    \end{proof}

\subsection{Proof of the main results.} 
    The scaling limits will be shown respectively in the discrete-time scenario and in the continuous-time scenario.
    \begin{proof}[Proof of the discrete-time Models.]
        We first show the functional CLT for the discrete-time stochastic Lévy--Lorentz gas $(X_n)_{n\geq0}$ and for the step-reinforced Lévy--Lorentz gas $(X^{(p)}_n)_{n\geq0}$ when $0\leq p<3/4$. Indeed, first notice that for any $n\in\mathbb{N}$, the càdlàg processes $(n^{-1/2}\omega_{n^{1/2}t})_{t\in\mathbb{R}}$ and $(n^{-1/2}S_{nt})_{t\in\mathbb{R}}$ are independent, where we set $S_{nt}\equiv0$ whenever $t<0$. Invoking Lemma \ref{lem: omega LLN}, we know $n^{-1/2}\omega_{n^{1/2}t}\to\ell t$ as $n\to\infty$, $\mathbb{P}$-a.s. On the other hand, the standard CLT for the Markovian random walk $(S_n)_{n\geq0}$ shows the following weak convergence $(n^{-1/2}(S_{nt}-E[S_1]nt))_{t\in\mathbb{R}}\Rightarrow (E[S_1^2]^{1/2}B_t)_{t\in\mathbb{R}}$ as $n\to\infty$ in the $J_1$-Skorokhod topology, where $(B_t)_{t\geq0}$ is the standard Brownian motion and $B_t\equiv0$ when $t<0$. Henceforth, we have the following weak convergence $\{(n^{-1/2}\omega_{n^{1/2}t})_{t\in\mathbb{R}},(n^{-1/2}(S_{nt}-E[S_1]nt))_{t\in\mathbb{R}}\}\Rightarrow\{(\ell t)_{t\in\mathbb{R}},(E[S_1^2]^{1/2}B_t)_{t\in\mathbb{R}}\}$ in $\mathcal{D}_2$ viewed from the $J_1$-Skorokhod topology. For each $n\in\mathbb{N}$, we now let $X_{nt}\equiv 0$ whenever $t<0$. Since $(\ell t)_{t\in\mathbb{R}}$ admits continuous and strictly increasing traces, by Lemma \ref{lem: skorokhod 3} we get the convergence $(X_{nt})_{t\in\mathbb{R}}=\mathcal{K}\{(n^{-1/2}\omega_{n^{1/2}t})_{t\in\mathbb{R}},(n^{-1/2}(S_{nt}-E[S_1]nt))_{t\in\mathbb{R}}\}\Rightarrow\mathcal{K}\{(\ell t)_{t\in\mathbb{R}},(E[S_1^2]^{1/2}B_t)_{t\in\mathbb{R}}\}=(\ell E[S_1^2]^{1/2}B_t)_{t\in\mathbb{R}}$ in $\mathcal{D}_1$ viewed from the $J_1$-Skorokhod topology. Hence the first assertion in Theorem \ref{thm: functional CLT} is verified. When it comes to the step-reinforced Lévy--Lorentz gas $(X^{(p)}_n)_{n\geq0}$, we use Lemma \ref{lem: martingale approach} as well as \cite[Theorem 4.3]{Chen/Laulin} for the asymptotic normality of the reinforced particle $(S^{(p)}_n)_{n\geq0}$, to proceed with an analogous argument.
    \end{proof}

    \begin{proof}[Proof of the continuous-time Models.] 
        We second show the functional CLT for the continuous-time stochastic Lévy--Lorentz gas $(\widetilde{X}_t)_{t\geq0}$ and for the step-reinforced Lévy--Lorentz gas $(\widetilde{X}^{(p)}_t)_{t\geq0}$ when $0\leq p<3/4$. Take $m\in\mathbb{N}$ and $0\leq s_1<\ldots<s_m<\infty$. From \textit{Part II.} of the proof of Theorem \ref{thm: regular CLT} we know that $X_{N_{s_k t}}-\ell E[S_1]N_{s_k t}\leq \widetilde{X}_{s_k t}-\ell E[S_1]s_k t\leq X_{N_{s_k t}+1}-\ell E[S_1](N_{s_k t}+1)$ for all $k=1,\ldots,m$. Notice that this inequality might also be the other way around. From the already shown weak convergence of the discrete-time model $(X_n)_{n\geq0}$, one has the convergence in distribution of the finite-dimensional $(N_{s_k t}^{-1/2}(X_{N_{s_kt}}-\ell E[S_1]N_{s_kt}))_{k=1,\ldots,m}$ to the desired multivariate Gaussian as $t\to\infty$. Here we also invoke Lemma \ref{lem: T and n ratio limit} which says $t^{-1}N_{s_k t}\to s_k (E[\abs{V_1}]\ell)^{-1}$ as $t\to\infty$, $\mathbb{P}\otimes P^S_0$-a.s. The proof for the finite-dimensional marginal convergence for the continuous-time step-reinforced gas is completely analogous, and we omit its details.
    \end{proof}
    To march from finite-dimensional marginals to convergence in distribution in $\mathcal{D}_1$, it suffices to verify tightness \cite{Billingsley} of the relevant processes. Indeed, the Prokhorov's theorem \cite{Prokhorov} should give us a tangible criterion for the tightness of the continuous-time processes from the discrete-time processes. We will not elaborate on this point, and we leave this open to future projects.

\bibliographystyle{plain}
\begin{spacing}{1}

\end{spacing}

\end{document}